\documentclass[12pt]{article}%
\usepackage[intlimits]{amsmath}
\usepackage{amssymb}
\usepackage[T1]{fontenc}
\usepackage[sc]{mathpazo}
\usepackage{color}
\usepackage[colorlinks]{hyperref}
\usepackage{amsfonts}
\usepackage{graphicx}%
\definecolor {refcol}{RGB}{40,0,255}
\hypersetup{colorlinks=true,allcolors=refcol}
\makeatletter
\newfont{\footsc}{cmcsc10 at 8truept}
\newfont{\footbf}{cmbx10 at 8truept}
\newfont{\footrm}{cmr10 at 10truept}
\newtheorem{theorem}{Theorem}

\newtheorem{conjecture}[theorem]{Conjecture}
\newtheorem{corollary}[theorem]{Corollary}

\newtheorem{problem}[theorem]{Problem}
\newtheorem{proposition}[theorem]{Proposition}

\newenvironment{proof}[1][Proof]{\noindent{\textbf {#1}  }}  {\hfill$\Box$\bigskip}
\begin{document}

\title{\textbf{Combinatorial methods for the spectral }$p$-\textbf{norm of
hypermatrices}}
\author{V. Nikiforov\thanks{Department of Mathematical Sciences, University of
Memphis, Memphis TN 38152, USA. Email: \textit{vnikifrv@memphis.edu}}}
\date{}
\maketitle

\begin{abstract}
The spectral $p$-norm of $r$-matrices generalizes the spectral $2$-norm of
$2$-matrices. In 1911 Schur gave an upper bound on the spectral $2$-norm of
$2$-matrices, which was extended in 1934 by Hardy, Littlewood, and P\'{o}lya
to $r$-matrices. Recently, Kolotilina, and independently the author,
strengthened Schur's bound for $2$-matrices. The main result of this paper
extends the latter result to $r$-matrices, thereby improving the result of
Hardy, Littlewood, and P\'{o}lya.

The proof is based on combinatorial concepts like $r$\emph{-partite }%
$r$\emph{-matrix} and \emph{symmetrant} of a matrix, which appear to be
instrumental in the study of the spectral $p$-norm in general. Thus, another
application shows that the spectral $p$-norm and the $p$-spectral radius of a
symmetric nonnegative $r$-matrix are equal whenever $p\geq r$. This result
contributes to a classical area of analysis, initiated by Mazur and Orlicz
back in 1930.

Additionally, a number of bounds are given on the $p$-spectral radius and the
spectral $p$-norm of $r$-matrices and $r$-graphs.\medskip



\end{abstract}

\section{Introduction}

In this paper we study the spectral $p$-norm of hypermatrices and its
applications to spectral hypergraph theory. Recall that the spectral $2$-norm
$\left\Vert A\right\Vert _{2}$ of an $m\times n$ matrix $A:=\left[
a_{i,j}\right]  $ is defined as
\[
\left\Vert A\right\Vert _{2}:=\max\text{ }\{|\sum_{i,j}a_{i,j}x_{j}%
y_{i}|\text{ : }\left\vert x_{1}\right\vert ^{2}+\cdots+\left\vert
x_{n}\right\vert ^{2}=1\text{ and }\left\vert y_{1}\right\vert ^{2}%
+\cdots+\left\vert y_{m}\right\vert ^{2}=1\}.
\]
Arguably, $\left\Vert A\right\Vert _{2}$ is the most important numeric
parameter of $A,$ so it is natural to compare it to other parameters of $A.$
One of the first results in this vein, given by Schur in \cite{Sch11}, p. 6,
reads as%
\begin{equation}
\left\Vert A\right\Vert _{2}^{2}\leq\max_{i\in\left[  m\right]  }r_{i}%
\max_{j\in\left[  n\right]  }c_{i}, \label{Schin}%
\end{equation}
where $r_{i}$ $:=\sum_{k\in\left[  n\right]  }\left\vert a_{i,k}\right\vert $
and $c_{j}:=\sum_{k\in\left[  m\right]  }\left\vert a_{k,j}\right\vert $. In
2006, Kolotilina \cite{Kol06} dramatically improved Schur's inequality,
showing that, in fact,
\begin{equation}
\left\Vert A\right\Vert _{2}^{2}\leq\max_{a_{i,j}\neq0}r_{i}c_{j}.
\label{ourin}%
\end{equation}
Sometimes inequality (\ref{ourin}) can be much stronger than (\ref{Schin});
e.g., if $A$ is the adjacency matrix of the star $K_{1,n},$ inequality
(\ref{Schin}) gives $\left\Vert A\right\Vert _{2}\leq n,$ while (\ref{ourin})
gives $\left\Vert A\right\Vert _{2}\leq\sqrt{n},$ which is best possible,
since $\left\Vert A\right\Vert _{2}=\sqrt{n}$. Let us add that an independent
and shorter proof of inequality (\ref{ourin}) was given also by the author in
\cite{Nik07}.\medskip

One of the goals of this paper is to extend inequality (\ref{ourin}) to
hypermatrices. Similar results can be traced back to Hardy, Littlewood, and
P\'{o}lya's book \textquotedblleft Inequalities\textquotedblright%
\ (\cite{HLP88}, p. 196), where Schur's inequality (\ref{Schin}) was extended
in several directions.\ To state an essential version of this result in terms
of hypermatrices, we introduce some terminology and notation.

Let $r\geq2,$ and let $n_{1},\ldots,n_{r}$ be positive integers. An
$r$\emph{-matrix} of order $n_{1}\times\cdots\times n_{r}$ is a function
defined on the Cartesian product $\left[  n_{1}\right]  \times\cdots
\times\left[  n_{r}\right]  .$ If $n_{1}=\cdots=n_{r}=n,$ then $A$ is called a
\emph{cubical} $r$-\emph{matrix} of order $n$\footnote{In graph theory the
order of a (hyper)graph is the number of its vertices, and in much of matrix
theory the order of a square matrix means the number of its rows. We keep
these meanings.}$,$ and $\left[  n\right]  $ is called the \emph{index set} of
$A.$

In this paper \textquotedblleft matrix\textquotedblright\ stands for
\textquotedblleft$r$-matrix\textquotedblright\ with unspecified $r$; thus,
ordinary\ matrices are referred to as $2$-matrices. Matrices are denoted by
capital letters, whereas their values are denoted by the corresponding
lowercase letters with the variables listed as subscripts. E.g., if $A$ is an
$r$-matrix, we let $a_{i_{1},\ldots,i_{r}}:=A\left(  i_{1},\ldots
,i_{r}\right)  $ for all admissible $i_{1},\ldots,i_{r}.$

Given an $r$-matrix $A,$ any of the $r!$ matrices obtained by permuting the
variables of $A$ is called a \emph{transpose }of $A$. A cubical matrix is
called \emph{symmetric} if all its transposes are identical.

Now, let $A$ be an $r$-matrix of order $n_{1}\times\cdots\times n_{r}.$ We say
that $A$ is a \emph{rank-one} matrix if there exist $r$ vectors
\[
\mathbf{x}^{\left(  1\right)  }:=(x_{1}^{\left(  1\right)  },\ldots,x_{n_{1}%
}^{\left(  1\right)  }),\ldots,\mathbf{x}^{\left(  r\right)  }:=(x_{1}%
^{\left(  r\right)  },\ldots,x_{n_{r}}^{\left(  r\right)  })
\]
such that $a_{i_{1},\ldots,i_{r}}=x_{i_{1}}^{\left(  1\right)  }\cdots
x_{i_{r}}^{\left(  r\right)  }$ for all $i_{1}\in\left[  n_{1}\right]
,\ldots,i_{r}\in\left[  n_{r}\right]  $.

Further, the \emph{linear form} of $A$ is a function $L_{A}:\mathbb{C}^{n_{1}%
}\times\cdots\times\mathbb{C}^{n_{r}}\rightarrow\mathbb{C}$ defined for any
vectors
\[
\mathbf{x}^{\left(  1\right)  }:=(x_{1}^{\left(  1\right)  },\ldots,x_{n_{1}%
}^{\left(  1\right)  }),\ldots,\mathbf{x}^{\left(  r\right)  }:=(x_{1}%
^{\left(  r\right)  },\ldots,x_{n_{r}}^{\left(  r\right)  })
\]
as
\[
L_{A}(\mathbf{x}^{\left(  1\right)  },\ldots,\mathbf{x}^{\left(  r\right)
}):=\sum_{i_{1},\ldots,i_{r}}a_{i_{1},\ldots,i_{r}}\overline{x_{i_{1}%
}^{\left(  1\right)  }}\cdots\overline{x_{i_{r}}^{\left(  r\right)  }}.
\]

A central problem in analysis is the study of the critical points of
$|L_{A}(\mathbf{x}^{\left(  1\right)  },\ldots,\mathbf{x}^{\left(  r\right)
})|$, subject to various constrains on $\mathbf{x}^{\left(  1\right)  }%
,\ldots,\mathbf{x}^{\left(  r\right)  }$. Thus, following Lim \cite{Lim05},
for any real number $p\geq1,$ define the\emph{ spectral }$p$\emph{-norm
}$\left\Vert A\right\Vert _{p}$ of $A$ as%
\begin{equation}
\left\Vert A\right\Vert _{p}:=\max\text{ }\{|L_{A}(\mathbf{x}^{\left(
1\right)  },\ldots,\mathbf{x}^{\left(  r\right)  })|\text{ : }|\mathbf{x}%
^{\left(  1\right)  }|_{p}=\cdots=|\mathbf{x}^{\left(  r\right)  }|_{p}=1\}.
\label{defs}%
\end{equation}
Here and further, $\left\vert \cdot\right\vert _{p}$ stands for the $l^{p}$
norm of vectors and matrices. Let us stress that our definition of $\left\Vert
A\right\Vert _{p}$ encompasses all real $p\geq1,$ a fact that implies numerous
subtle consequences. It is worth pointing out that the dual $\left\Vert
A\right\Vert _{\ast}$ of the norm $\left\Vert A\right\Vert _{2}$ is called the
\emph{nuclear norm} of $A$\footnote{The nuclear norm is fundamental for tensor
products of Banach spaces. It was introduced quite a while ago by Schatten
\cite{Sch50} and Grothendieck \cite{Gro55}, but is enjoying a renewed interest
presently; see, e.g., \cite{Der16}, \cite{FrLi16}, \cite{FrLi16a},
\cite{ZLi15}, and \cite{LiCo14}.}.

Next, we generalize the rows and columns of $2$-matrices: Let $A$ be an
$r$-matrix of order $n_{1}\times\cdots\times n_{r}.$ For any $k\in\left[
r\right]  $ and $s\in\left[  n_{k}\right]  ,$ the $\left(  r-1\right)
$-matrix $A_{s}^{\left(  k\right)  }$ obtained from $A$ by fixing $i_{k}=s$ is
called a \emph{slice }of $A.$ E.g., if $r=2,$ then $A_{1}^{\left(  1\right)
},\ldots,A_{n_{1}}^{\left(  1\right)  }$ are the rows of $A$ and
$A_{1}^{\left(  2\right)  },\ldots,A_{n_{2}}^{\left(  2\right)  }$ are its
columns. The dual concept of a slice is called a \emph{fiber, }defined as a
vector obtained by fixing all but one variables of $A.$

Further, for every $k\in\left[  r\right]  ,$ let
\[
S^{\left(  k\right)  }:=\max\{|A_{1}^{\left(  k\right)  }|_{1},\ldots
,|A_{n_{k}}^{\left(  k\right)  }|_{1}\}.
\]
Now, an essential version of the result of Hardy, Littlewood, and P\'{o}lya
(\cite{HLP88}, Theorem 273) reads as:\medskip

\textbf{Theorem }\emph{For any }$r$\emph{-matrix }$A,$
\begin{equation}
\left\Vert A\right\Vert _{r}^{r}\leq S^{\left(  1\right)  }\cdots S^{\left(
r\right)  }. \label{HLP}%
\end{equation}
Clearly, inequality (\ref{HLP}) extends Schur's bound to $r$-matrices, but it
does not extend Kolotilina's inequality (\ref{ourin}) in any way. Thus, we
propose a bound that extends (\ref{ourin}) and strengthens (\ref{HLP}):

\begin{theorem}
\label{mth}If $A$ is an $r$-matrix of order $n_{1}\times\cdots\times n_{r},$
with slices $A_{s}^{\left(  k\right)  }$ $(k\in\left[  r\right]  ,$
$s\in\left[  n_{k}\right]  ),$ then%
\begin{equation}
\left\Vert A\right\Vert _{r}^{r}\leq\max_{a_{i_{1},\ldots,i_{r}}\neq
0}|A_{i_{1}}^{\left(  1\right)  }|_{1}\cdots|A_{i_{r}}^{\left(  r\right)
}|_{1}. \label{main}%
\end{equation}

\end{theorem}

Once again, the adjacency matrices of $\beta$-stars\footnote{$\beta$-stars are
hypergraphs such that all edges share the same vertex and no two edges share
other vertices.} (see \cite{Ber87}, p. 116) show that (\ref{main}) can be
essentially better than (\ref{HLP}).

Our proof of Theorem \ref{mth} turned out to be much more difficult than the
proof of (\ref{ourin}), and needs a multistage preparation, starting basically
from scratch. To explain the main idea of our approach, let $A$ be a real
$2$-matrix and set%
\begin{equation}
B:=\left[
\begin{array}
[c]{cc}%
0 & A\\
A^{T} & 0
\end{array}
\right]  .\label{sym2}%
\end{equation}
It is known (see, e.g., \cite{HoJo88}, p. 418) that the nonzero singular
values of $A$ are precisely the positive eigenvalues of \ $B$; in particular,
$\left\Vert A\right\Vert _{2}$ is the maximal eigenvalue of $B$. The matrix
$B$ has been extended to $r$-matrices by Ragnarsson and Van Loan in
\cite{RaVL13}. We establish some properties of this extension, eventually
obtaining a proof of Theorem \ref{mth}. These results turn out to be useful
also for other problems; in particular, to make some progress in a classical
area of analysis started by Mazur and Orlicz around 1930, see \cite{Sco81}, p
143. To convey the gist of this topic, we need two more definitions:

Given a real symmetric matrix $A$ of order $n,$ the \emph{polynomial form}
$P_{A}$ of $A$ is a function $P_{A}:\mathbb{R}^{n}\rightarrow\mathbb{R}$
defined for any vector $\mathbf{x}:=\left(  x_{1},\ldots,x_{n}\right)  $ as
\[
P_{A}\left(  \mathbf{x}\right)  :=L_{A}\left(  \mathbf{x},\ldots
,\mathbf{x}\right)  =\sum_{i_{1},\ldots,i_{r}}a_{i_{1},\ldots,i_{r}}x_{i_{1}%
}\cdots x_{i_{r}}.
\]
Further, for any real $p\geq1$, define the $p$\emph{-spectral radius }%
$\eta^{\left(  p\right)  }\left(  A\right)  $ of $A$ as%
\[
\eta^{\left(  p\right)  }\left(  A\right)  :=\max\{|P_{A}\left(
\mathbf{x}\right)  |:\mathbf{x}\in\mathbb{R}^{n}\text{ and }\left\vert
\mathbf{x}\right\vert _{p}=1\}.
\]

Clearly, for any $p\geq1$, we have $\eta^{\left(  p\right)  }\left(  A\right)
\leq\left\Vert A\right\Vert _{p},$ but equality rarely holds. In fact, the
relations between $\eta^{\left(  p\right)  }\left(  A\right)  $ and
$\left\Vert A\right\Vert _{p}$ have been studied for almost nine decades by
now, albeit under no special names. In particular, motivated by problem 73 of
Mazur and Orlicz in \cite{Sco81}, Banach \cite{Ban38} proved a result that
implies the following basic fact:\medskip

\textbf{Theorem }\emph{If }$A$\emph{ is a real symmetric }$r$\emph{-matrix of
order} $n$\emph{, then}
\[
\eta^{\left(  2\right)  }\left(  A\right)  =\left\Vert A\right\Vert _{2}.
\]
For newer proofs of Banach's result, see \cite{Fri13} and \cite{PST07}, and
for further results, see \cite{Din99} and \cite{PST07}, and their references.
Nonetheless, this area still holds surprises, as seen in the following theorem:

\begin{theorem}
\label{mth1}If $A$ is a nonnegative symmetric $r$-matrix and $p\geq r$, then%
\[
\eta^{\left(  p\right)  }\left(  A\right)  =\left\Vert A\right\Vert _{p}.
\]

\end{theorem}

As mentioned above, Theorem \ref{mth1} is proved in the same combinatorial
framework as Theorem \ref{mth}, but its proof also needs a supporting
Perron-Frobenius mini-theory. It is possible that the following stronger
assertion holds:

\begin{conjecture}
For every integer $r\geq3,$ there is a $p_{0}\left(  r\right)  \in\left(
1,2\right)  $ such that if $A$ is a nonnegative symmetric $r$-matrix, then
$\eta^{\left(  p\right)  }\left(  A\right)  =\left\Vert A\right\Vert _{p}$ for
every $p\geq p_{0}\left(  r\right)  .$
\end{conjecture}

The remaining part of the paper is split into three sections: in Section
\ref{cols}, we present relevant definitions and lay the basis for the proofs
of Theorems \ref{mth} and \ref{mth1}. Section \ref{boms} contains the proof of
Theorem \ref{mth} and several bounds on the spectral $p$-norm, conceived in
the spirit of spectral graph theory. Finally, Section \ref{gras} lists some
consequences of the main theorems for hypergraphs.

\section{\label{cols}Collecting some spectral tools for hypermatrices}

In this section we assemble the machinery that is needed for the proofs of
Theorems \ref{mth} and \ref{mth1}. Although these proofs are our primary goal,
we address broader topics involving $\left\Vert A\right\Vert _{p}$ and
$\eta^{\left(  p\right)  }\left(  A\right)  ,$ and explore many sidetracks as
well. The section is rather long, so we outline its main topics first:

Section \ref{ele} presents elemental properties of $\left\Vert A\right\Vert
_{p}$ as a function of $p$. In Section \ref{syms}, we discuss real symmetric
matrices and their polynomial forms. Section \ref{eigs} presents the basics on
eigenequations. In Section \ref{PFs}, we build a Perron-Frobenius mini-theory,
since the existing Perron-Frobenius theory of nonnegative hypermatrices is not
sufficient for the proofs of Theorem \ref{mth} and \ref{mth1}. Sections
\ref{rpar} and \ref{symas} introduce the new concepts \textquotedblleft%
$r$-partite $r$-matrix\textquotedblright\ and \textquotedblleft
symmetrant\textquotedblright, both of which are crucial for our study of
$\left\Vert A\right\Vert _{p}$. Section \ref{symas} concludes with the proof
of Theorem \ref{mth1}.

\subsection{\label{ele}Basic properties of $\left\Vert A\right\Vert _{p}$}

Let $A$ be an $r$-matrix of order $n_{1}\times\cdots\times n_{r}.$ If the
vectors $\mathbf{x}^{\left(  1\right)  }\in\mathbb{C}^{n_{1}},\ldots
,\mathbf{x}^{\left(  r\right)  }\in\mathbb{C}^{n_{r}}$ satisfy $|\mathbf{x}%
^{\left(  1\right)  }|_{p}=\cdots=|\mathbf{x}^{\left(  r\right)  }|_{p}=1$
and
\[
\left\Vert A\right\Vert _{p}=|L_{A}(\mathbf{x}^{\left(  1\right)  }%
,\ldots,\mathbf{x}^{\left(  r\right)  })|,
\]
the $r$-tuple $\mathbf{x}^{\left(  1\right)  },\ldots,\mathbf{x}^{\left(
r\right)  }$ is called an \emph{eigenkit }to $\left\Vert A\right\Vert _{p}.$

Since $\left\Vert A\right\Vert _{p}$ is defined for every real $p\geq1,$ it is
useful to investigate the function $h_{A}\left(  x\right)  =\left\Vert
A\right\Vert _{x}$ for fixed $A$ and $x\geq1$. Set $\left\vert A\right\vert
_{\max}=\max_{i_{1},\ldots,i_{r}}|a_{i_{1},\ldots,i_{r}}|,$ and note a few
properties of \ $\left\Vert A\right\Vert _{p}$:

\begin{proposition}
\label{pro2}If $A$ is an $r$-matrix of order $n_{1}\times\cdots\times n_{r},$
then $\left\Vert A\right\Vert _{p}$ has the following properties:

(a) $\left\Vert A\right\Vert _{1}=|A|_{\max}$;

(b) If $p\geq q\geq1,$ then $\left\Vert A\right\Vert _{p}\geq\left\Vert
A\right\Vert _{q}$;

(c) If $p\geq q\geq1,$ then $\left(  n_{1}\cdots n_{r}\right)  ^{1/p}%
\left\Vert A\right\Vert _{p}\leq\left(  n_{1}\cdots n_{r}\right)
^{1/q}\left\Vert A\right\Vert _{q}$;

(d) $\left\Vert A\right\Vert _{p}\leq\left\vert A\right\vert _{1}$;

(e) $\left\Vert A\right\Vert _{p}$ is Lipschitz continuous in $p$, that is, if
$p\geq q\geq0$, then%
\[
0\leq\left\Vert A\right\Vert _{p}-\left\Vert A\right\Vert _{q}\leq\left(
p-q\right)  \left\vert A\right\vert _{1}\left(  n_{1}\cdots n_{r}\right)
\log\left(  n_{1}\cdots n_{r}\right)  .
\]

\end{proposition}

\begin{proof}
\emph{(a)} Let $\left\vert A\right\vert _{\max}=|a_{i_{1},\ldots,i_{r}}|,$ and
for any $k\in\left[  r\right]  ,$ let $\mathbf{x}^{\left(  k\right)  }%
:=(x_{1}^{\left(  k\right)  },\ldots,x_{n_{k}}^{\left(  k\right)  })$ be with
$x_{i_{k}}^{\left(  k\right)  }=1$ and zero elsewhere. Hence, $\left\vert
A\right\vert _{\max}=|L_{A}(\mathbf{x}^{\left(  1\right)  },\ldots
,\mathbf{x}^{\left(  r\right)  })|~\leq\left\Vert A\right\Vert _{1}.$ On the
other hand,
\[
\left\Vert A\right\Vert _{1}\leq|\sum_{i_{1},\ldots,i_{r}}a_{i_{1}%
,\ldots,i_{r}}\overline{x_{i_{1}}^{\left(  1\right)  }}\cdots\overline
{x_{i_{r}}^{\left(  r\right)  }}|~\leq|A|_{\max}||\mathbf{x}^{\left(
1\right)  }|_{1}\cdots|\mathbf{x}^{\left(  r\right)  }|_{1}=|A|_{\max}.
\]

\emph{(b)} Let $\mathbf{x}^{\left(  1\right)  },\ldots,\mathbf{x}^{\left(
r\right)  }$ be an eigenkit to $\left\Vert A\right\Vert _{q}.$ Since the
entries of $\mathbf{x}^{\left(  1\right)  },\ldots,\mathbf{x}^{\left(
r\right)  }$ are of modulus at most one and $p/q\geq1$, it turns out that
$|\mathbf{x}^{\left(  k\right)  }|_{q}\geq|\mathbf{x}^{\left(  k\right)
}|_{p}$ for any $k\in\left[  r\right]  .$ Hence, $|\mathbf{x}^{\left(
1\right)  }|_{p}\leq1,\ldots,|\mathbf{x}^{\left(  r\right)  }|_{p}\leq1,$ and
so%
\begin{align*}
\left\Vert A\right\Vert _{q}  &  =|L_{A}(\mathbf{x}^{\left(  1\right)
},\ldots,\mathbf{x}^{\left(  r\right)  })|~\leq\frac{1}{|\mathbf{x}^{\left(
1\right)  }|_{p}\cdots|\mathbf{x}^{\left(  r\right)  }|_{p}}|L_{A}%
(\mathbf{x}^{\left(  1\right)  },\ldots,\mathbf{x}^{\left(  r\right)  })|\\
&  =|L_{A}(\frac{1}{|\mathbf{x}^{\left(  1\right)  }|_{p}}\mathbf{x}^{\left(
1\right)  },\ldots,\frac{1}{|\mathbf{x}^{\left(  r\right)  }|_{p}}%
\mathbf{x}^{\left(  r\right)  })|~\leq\left\Vert A\right\Vert _{p}.
\end{align*}

\emph{(c) }Let $\mathbf{x}^{\left(  1\right)  },\ldots,\mathbf{x}^{\left(
r\right)  }$ be an eigenkit to $\left\Vert A\right\Vert _{p}.$ The Power Mean
inequality implies that%
\[
|\mathbf{x}^{\left(  k\right)  }|_{q}\leq n_{k}^{1/q-1/p}|\mathbf{x}^{\left(
k\right)  }|_{q}=n_{k}^{1/q-1/p}%
\]
for any $k\in\left[  r\right]  .$ Now, for any $k\in\left[  r\right]  ,$ set%
\[
\mathbf{y}^{\left(  k\right)  }:=n_{k}^{1/p-1/q}\mathbf{x}^{\left(  k\right)
},
\]
and note that $|\mathbf{y}^{\left(  k\right)  }|_{q}\leq1.$ Therefore,
\begin{align*}
\left\Vert A\right\Vert _{p}  &  =|L_{A}(\mathbf{x}^{\left(  1\right)
},\ldots,\mathbf{x}^{\left(  r\right)  })|~=\left(  n_{1}\cdots n_{r}\right)
^{1/q-1/p}|L_{A}(\mathbf{y}^{\left(  1\right)  },\ldots,\mathbf{y}^{\left(
r\right)  })|\\
&  \leq\left(  n_{1}\cdots n_{r}\right)  ^{1/q-1/p}\left\Vert A\right\Vert
_{q}.
\end{align*}

\emph{(d) }Let $\mathbf{x}^{\left(  1\right)  },\ldots,\mathbf{x}^{\left(
r\right)  }$ be an eigenkit to $\left\Vert A\right\Vert _{q}$. We see that
\begin{align*}
\left\Vert A\right\Vert _{p}  &  =|L_{A}(\mathbf{x}^{\left(  1\right)
},\ldots,\mathbf{x}^{\left(  r\right)  })|~=|\sum_{i_{1},\ldots,i_{r}}%
a_{i_{1},\ldots,i_{r}}\overline{x_{i_{1}}^{\left(  1\right)  }}\cdots
\overline{x_{i_{r}}^{\left(  r\right)  }}|\\
&  \leq\sum_{i_{1},\ldots,i_{r}}|a_{i_{1},\ldots,i_{r}}||x_{i_{1}}^{\left(
1\right)  }|\cdots|x_{i_{r}}^{\left(  r\right)  }|~\leq\sum_{i_{1}%
,\ldots,i_{r}}|a_{i_{1},\ldots,i_{r}}|~=\left\vert A\right\vert _{1}.
\end{align*}

\emph{(e)} Let $p>q\geq1$. Clauses \emph{(c)} and \emph{(d) }imply that
\begin{align*}
\left\Vert A\right\Vert _{p}-\left\Vert A\right\Vert _{q}  &  \leq\left(
n_{1}\cdots n_{r}\right)  ^{1/q-1/p}\left\Vert A\right\Vert _{q}-\left\Vert
A\right\Vert _{q}=\left\Vert A\right\Vert _{q}\frac{\left(  n_{1}\cdots
n_{r}\right)  ^{1/q}-\left(  n_{1}\cdots n_{r}\right)  ^{1/p}}{\left(
n_{1}\cdots n_{r}\right)  ^{1/p}}\\
&  \leq|A|_{1}(\left(  n_{1}\cdots n_{r}\right)  ^{1/q}-\left(  n_{1}\cdots
n_{r}\right)  ^{1/p}).
\end{align*}
Now, the Mean Value Theorem applied to the function $f\left(  x\right)
:=\left(  n_{1}\cdots n_{r}\right)  ^{1/x}$ implies that there exists a
$\theta\in\left(  p,q\right)  $ such that%
\begin{align*}
\left(  n_{1}\cdots n_{r}\right)  ^{1/p}-\left(  n_{1}\cdots n_{r}\right)
^{1/q}  &  =\left(  p-q\right)  f^{\prime}\left(  \theta\right) \\
&  =-\left(  p-q\right)  \left(  n_{1}\cdots n_{r}\right)  ^{1/\theta}%
\theta^{-2}\log\left(  n_{1}\cdots n_{r}\right)  .
\end{align*}
In view of $\left(  n_{1}\cdots n_{r}\right)  ^{1/\theta}\theta^{-2}%
<n_{1}\cdots n_{r},$ the required inequality follows, completing the proof of
Proposition \ref{pro2}.
\end{proof}

Since $\left\Vert A\right\Vert _{p}$ is nondecreasing and bounded in $p$, the
limit $\lim_{p\rightarrow\infty}\left\Vert A\right\Vert _{p}$ exists. It is
not hard to see that if $A$ is nonnegative, then $\lim_{p\rightarrow\infty
}\left\Vert A\right\Vert _{p}=\left\vert A\right\vert _{1}$, but in general
the value of this limit is not clear, so we raise a problem:

\begin{problem}
Find $\lim_{p\rightarrow\infty}\left\Vert A\right\Vert _{p}$ for any matrix
$A.$
\end{problem}

We do not know if $\left\Vert A\right\Vert _{p}$ is differentiable in $p$, so
we conclude this subsection with another open problem:

\begin{problem}
Is the function $\left\Vert A\right\Vert _{p}$ piecewise differentiable in $p$
for any matrix $A$?
\end{problem}

\subsection{\label{syms}Real symmetric matrices}

Given an $n$-vector $\mathbf{x}$ and a set $X\subset\left[  n\right]  ,$ write
$\mathbf{x}|_{X}$ for the restriction of $\mathbf{x}$ over the set $X$ in the
order induced by $\left[  n\right]  $. Further, for any real number $p\geq1$,
write $\mathbb{S}_{p}^{n-1}$ for the set of all real vectors $\left(
x_{1},\ldots,x_{n}\right)  $ with $\left\vert x_{1}\right\vert ^{p}%
+\cdots+\left\vert x_{n}\right\vert ^{p}=1$

Let $A$ be a real symmetric $r$-matrix $A$ of order $n$. Define
\begin{align*}
\lambda^{\left(  p\right)  }\left(  A\right)   &  :=\max\{P_{A}\left(
\mathbf{x}\right)  :\mathbf{x}\in\mathbb{S}_{p}^{n-1}\},\\
\lambda_{\min}^{\left(  p\right)  }\left(  A\right)   &  :=\min\{P_{A}\left(
\mathbf{x}\right)  :\mathbf{x}\in\mathbb{S}_{p}^{n-1}\}.
\end{align*}
Since $\eta^{\left(  p\right)  }\left(  A\right)  =\max\{\left\vert
P_{A}\left(  \mathbf{x}\right)  \right\vert :\mathbf{x}\in\mathbb{S}_{p}%
^{n-1}\},$ we see that
\[
\eta^{\left(  p\right)  }\left(  A\right)  :=\max\{|\lambda^{\left(  p\right)
}\left(  A\right)  |,|\lambda_{\min}^{\left(  p\right)  }\left(  A\right)
|\}.
\]

The values $\lambda^{\left(  r\right)  }\left(  A\right)  $, $\lambda_{\min
}^{\left(  r\right)  }\left(  A\right)  $, and $\eta^{\left(  r\right)
}\left(  A\right)  $ are particular for the $r$-matrix $A$; thus, for
simplicity we set $\lambda\left(  A\right)  :=\lambda^{\left(  r\right)
}\left(  A\right)  ,$ $\lambda_{\min}\left(  A\right)  :=\lambda_{\min
}^{\left(  r\right)  }\left(  A\right)  ,$ and $\eta\left(  A\right)
:=\eta^{\left(  r\right)  }\left(  A\right)  $. In particular, if $r=2,$ then
$\lambda\left(  A\right)  $ and $\lambda_{\min}\left(  A\right)  $ are the
largest and smallest eigenvalues of $A,$ and $\eta\left(  A\right)  $ is the
spectral radius of $A.$ For graphs and hypergraphs the values $\lambda
^{\left(  p\right)  }\left(  A\right)  $ and $\lambda_{\min}^{\left(
p\right)  }\left(  A\right)  $ have been extensively studied---see
\cite{Nik14} and its references.

Let $p\geq1,$ and let $\lambda\in\{\lambda^{\left(  p\right)  }\left(
A\right)  ,\lambda_{\min}^{\left(  p\right)  }\left(  A\right)  \}.$ A vector
$\mathbf{x}\in$ $\mathbb{S}_{p}^{n-1}$ such that $\lambda=P_{A}\left(
\mathbf{x}\right)  $ is called an \emph{eigenvector }to $\lambda.$ Note that
if $p\neq r,$ the norms of the eigenvectors to $\lambda^{\left(  p\right)
}\left(  A\right)  $ and $\lambda_{\min}^{\left(  p\right)  }\left(  A\right)
$ are essential for their definition. If $\mathbf{x}\in$ $\mathbb{S}_{p}%
^{n-1}$ and $\eta^{\left(  p\right)  }\left(  A\right)  =$ $\left\vert
P_{A}\left(  \mathbf{x}\right)  \right\vert ,$ for convenience we say that
$\mathbf{x}$ is an \emph{eigenvector} to $\eta^{\left(  p\right)  }\left(
A\right)  .$

Note two fundamental identities about the polynomial form of $A$:
\begin{align}
\frac{dP_{A}\left(  \mathbf{x}\right)  }{dx_{k}}  &  =r\sum_{i_{2}%
,\ldots,i_{r}}a_{k,i_{2},\ldots,i_{r}}x_{i_{2}}\cdots x_{i_{r}},\text{
}k=1,\ldots,n,\label{mid}\\
\sum_{k\in\left[  n\right]  }\frac{\partial P_{A}\left(  \mathbf{x}\right)
}{\partial x_{k}}x_{k}  &  =rP_{A}\left(  \mathbf{x}\right)  . \label{mid1}%
\end{align}

For the sake of applications, it is useful to investigate the function
$h_{A}\left(  x\right)  =\eta^{\left(  x\right)  }\left(  A\right)  $ for
fixed symmetric matrix $A$ and $x\geq0$. Here we state a few properties of
$\lambda^{\left(  p\right)  }\left(  A\right)  ,$ which can be proved as in
Proposition \ref{pro2}:

\begin{proposition}
\label{pro3}If $A$ is a real symmetric $r$-matrix of order $n,$ then
$\eta^{\left(  p\right)  }\left(  A\right)  $ has the following properties:

(a) $\eta^{\left(  1\right)  }\left(  A\right)  \geq|A|_{\max}r!/r^{r}$;

(b) If $p\geq q\geq1,$ then $\eta^{\left(  p\right)  }\left(  A\right)
\geq\eta^{\left(  q\right)  }\left(  A\right)  $;

(c) If $p\geq q\geq1,$ then $n^{r/p}\eta^{\left(  p\right)  }\left(  A\right)
\leq n^{r/q}\eta^{\left(  q\right)  }\left(  A\right)  $;

(d) $\eta^{\left(  p\right)  }\left(  A\right)  \leq\left\vert A\right\vert
_{1}$;

(e) $\eta^{\left(  p\right)  }\left(  A\right)  $ is Lipschitz continuous in
$p.$ If $p>q\geq1,$ then%
\[
0\leq\eta^{\left(  p\right)  }\left(  A\right)  -\eta^{\left(  q\right)
}\left(  A\right)  \leq\left(  p-q\right)  \left\vert A\right\vert _{1}%
n^{r}\log\left(  n^{r}\right)  .
\]

\end{proposition}

\subsection{\label{eigs}Eigenequations}

Let $A$ be a cubical $r$-matrix of order $n.$ Following \cite{Qi05} and
\cite{CPZ08}, we say that a complex number $\lambda$ is an \emph{eigenvalue}
of $A$ if%
\begin{equation}
\lambda x_{k}^{r-1}=\sum_{i_{2},\ldots,i_{r}}a_{k,i_{2},\ldots,i_{r}}x_{i_{2}%
}\cdots x_{i_{r}}\ \ \ \ k=1,\ldots,n, \label{alg}%
\end{equation}
for some nonzero complex vector $\left(  x_{1},\ldots,x_{n}\right)  $, called
an \emph{eigenvector} to $\lambda.$

Recently, eigenvalues of $r$-matrices have been studied intensively and have
been put on a solid ground (see, e.g., \cite{Qi05} and \cite{HHLQ13}). We
shall not need this whole theory except the concept of spectral radius. Recall
that the \emph{spectral radius} $\rho\left(  A\right)  $ of a cubical matrix
$A$ is the largest modulus of an eigenvalue of $A$. As we shall see, if $A$ is
a symmetric nonnegative matrix, then
\[
\eta^{\left(  r\right)  }\left(  A\right)  =\rho\left(  A\right)  .
\]
However, if $r>2$, this identity may not hold for arbitrary real symmetric $r$-matrices.

Next, we show that $\lambda^{\left(  p\right)  }\left(  A\right)  $ and
$\lambda_{\min}^{\left(  p\right)  }\left(  A\right)  $ satisfy a system of
equations similar to (\ref{alg}). Suppose that $A$ is a real symmetric
$r$-matrix of order $n$ and let $\mathbf{x}:=\left(  x_{1},\ldots
,x_{n}\right)  \in\mathbb{S}_{p}^{n-1}$ be an eigenvector to $\lambda^{\left(
p\right)  }\left(  G\right)  $. If $p>1,$ the function $\left\vert
x_{1}\right\vert ^{p}+\ldots+~\left\vert x_{n}\right\vert ^{p}$ has continuous
derivatives in each variable $x_{i}$. Thus, Lagrange's method implies that
there exists a $\mu$ such that for each $k=1,\ldots,n,$
\[
\mu\frac{\partial(\left\vert x_{1}\right\vert ^{p}+\ldots+~\left\vert
x_{n}\right\vert ^{p})}{\partial x_{k}}=\mu px_{k}\left\vert x_{k}\right\vert
^{p-2}=\frac{\partial P_{A}\left(  \mathbf{x}\right)  }{\partial x_{k}}.
\]
Now, multiplying the $k$th equation by $x_{k}$ and adding all equations, we
find that
\[
\mu p=\mu p\sum_{k\in\left[  n\right]  }\left\vert x_{k}\right\vert ^{p}%
=\sum_{k\in\left[  n\right]  }\frac{\partial P_{A}\left(  \mathbf{x}\right)
}{\partial x_{k}}x_{k}=rP_{A}\left(  \mathbf{x}\right)  =r\lambda^{\left(
p\right)  }\left(  G\right)  .
\]
Hence, we arrive at the following theorem:

\begin{theorem}
Let $A$ be real symmetric $r$-matrix of order $n$ and let $p>1.$ If
$\lambda\in\{\lambda^{\left(  p\right)  }\left(  G\right)  ,\lambda_{\min
}^{\left(  p\right)  }\left(  G\right)  \}$ and $\mathbf{x}:=\left(
x_{1},\ldots,x_{n}\right)  \in\mathbb{S}_{p}^{n-1}$ is an eigenvector to
$\lambda,$ then $x_{1},\ldots,x_{n}$ satisfy the equations%
\begin{equation}
\lambda x_{k}\left\vert x_{k}\right\vert ^{p-2}=\frac{1}{r}\frac{\partial
P_{A}\left(  \mathbf{x}\right)  }{\partial x_{k}},\text{ \ \ \ \ }%
k=1,\ldots,n. \label{eequ}%
\end{equation}

\end{theorem}

Starting with equations (\ref{eequ}), it is possible to introduce a new class
of eigenvalues, but we do not pursue this direction here, except for a simple
proposition needed for the proof of Theorem \ref{th2}. Recall that in
\cite{Qi13}, Qi showed that if $A$ is a symmetric nonnegative $r$-matrix, then
$\eta^{\left(  r\right)  }\left(  A\right)  $ is an eigenvalue of $A$ of
largest modulus, that is to say, $\eta^{\left(  r\right)  }\left(  A\right)
=\rho\left(  A\right)  $. It turns out that a similar statement holds for
$\eta^{\left(  p\right)  }\left(  A\right)  $ for any $p>1:$

\begin{proposition}
\label{pro}Let $A$ be a symmetric nonnegative $r$-matrix and let $p>1.$ If
$\lambda\in\mathbb{R}$ and $\mathbf{x}:=\left(  x_{1},\ldots,x_{n}\right)
\in\mathbb{S}_{p}^{n-1}$ satisfy the equations
\begin{equation}
\lambda x_{k}\left\vert x_{k}\right\vert ^{p-2}=\frac{1}{r}\frac{\partial
P_{A}\left(  \mathbf{x}\right)  }{\partial x_{k}},\text{ \ \ \ \ }%
k=1,\ldots,n. \label{spe}%
\end{equation}
then $\left\vert \lambda\right\vert \leq\eta^{\left(  p\right)  }\left(
A\right)  .$
\end{proposition}

\begin{proof}
If $\mathbf{x}:=\left(  x_{1},\ldots,x_{n}\right)  \in\mathbb{S}_{p}^{n-1}$
and $\lambda$ satisfy the equations (\ref{spe}), then
\[
\left\vert \lambda\right\vert \left\vert x_{k}\right\vert ^{p}=\frac{1}%
{r}\left\vert \frac{\partial P_{A}\left(  \mathbf{x}\right)  }{\partial x_{k}%
}\right\vert \left\vert x_{k}\right\vert \leq\sum_{i_{2},\ldots,i_{r}%
}a_{k,i_{2},\ldots,i_{r}}|x_{k}|\left\vert x_{i_{2}}\right\vert \cdots
\left\vert x_{i_{r}}\right\vert .
\]
Adding all inequalities together, we find that
\[
\left\vert \lambda\right\vert =\left\vert \lambda\right\vert \sum_{k\in\left[
r\right]  }\left\vert x_{k}\right\vert ^{p}\leq P_{A}\left(  |x_{1}%
|,\ldots,|x_{r}|\right)  \leq\eta^{\left(  p\right)  }\left(  A\right)  ,
\]
completing the proof.
\end{proof}

\subsection{\label{PFs}A Perron-Frobenius mini-theory}

The combined work of Chang, Pearson, and Zhang \cite{CPZ08}, Friedland,
Gaubert, and Han \cite{FGH11}, and Yang and Yang \cite{YaYa11} laid the ground
for a Perron-Frobenius theory of nonnegative hypermatrices. Roughly speaking
this theory studies $\rho\left(  A\right)  $ of nonnegative cubical
$r$-matrices and its eigenvectors. However, for symmetric nonnegative
$r$-matrices the parameter $\eta^{\left(  p\right)  }\left(  A\right)  $ is
more general than $\rho\left(  A\right)  $, and the existing Perron-Frobenius
theory does not cover $\eta^{\left(  p\right)  }\left(  A\right)  $ for $p\neq
r$. Thus, in this section we give some new Perron-Frobenius type theorems,
which are necessary for our proofs.

Note that if $A$ is a symmetric nonnegative matrix, then $\eta^{\left(
p\right)  }\left(  A\right)  =\lambda^{\left(  p\right)  }\left(  A\right)  ,$
so these two values can be used interchangeably.

The \emph{digraph} $\mathcal{D}\left(  A\right)  $ of a cubical $r$-matrix of
order $n$ is defined by setting $V\left(  \mathcal{D}\left(  A\right)
\right)  :=\left[  n\right]  $ and letting $\left\{  k,j\right\}  \in E\left(
\mathcal{D}\left(  A\right)  \right)  $ whenever there is a nonzero entry
$a_{k,i_{2},\ldots,i_{r}}$ such that $j\in\left\{  i_{2},\ldots,i_{r}\right\}
$. Following \cite{FGH11}, a cubical matrix is called \emph{weakly
irreducible} if its digraph is strongly connected; if a cubical matrix is not
weakly irreducible, it is called \emph{weakly reducible}.

In analogy to $2$-matrices, given a cubical $r$-matrix $A$ of order $n$ and a
set $X\subset\left[  n\right]  $, we write $A\left[  X\right]  $ for the
cubical matrix that is the restriction of $A$ over $X,$ and call $A\left[
X\right]  $ a \emph{principal} submatrix of $A$ \emph{induced} by $X.$

Clearly the digraph $\mathcal{D}\left(  A\right)  $ of a symmetric matrix $A$
is an undirected $2$-graph. If $A$ is a weakly reducible symmetric matrix,
then $\mathcal{D}\left(  A\right)  $ is disconnected and the vertices of each
component of $\mathcal{D}\left(  A\right)  $ induce a weakly irreducible
principal submatrix of $A,$ called a \emph{component }of $A$. Obviously, a
symmetric matrix is a\ block diagonal matrix of its components.

Our first theorem is typical for this area, but still holds a small surprise,
because it is valid for $p>r-1,$ whereas all known similar statements require
that $p=r.$ We find this fact a vindication for the study of $\eta^{\left(
p\right)  }\left(  A\right)  $ for any real $p\geq1.$

\begin{theorem}
\label{PF0}Let $r\geq2,$ $p>r-1,$ $A$ be a symmetric nonnegative $r$-matrix,
and $\mathbf{x}$ be a nonnegative eigenvector to $\lambda^{\left(  p\right)
}\left(  A\right)  .$ If $A$ is weakly irreducible, then $\mathbf{x}$ is positive.
\end{theorem}

\begin{proof}
Assume for a contradiction that $\mathbf{x}:=\left(  x_{1},\ldots
,x_{n}\right)  $ has zero entries, and set \ $Z:=\left\{  i:x_{i}=0\right\}
$. Since $A$ is weakly irreducible, there exist $i_{1},\ldots,i_{r}$ such that
$a_{i_{1},\ldots,i_{r}}>0$ and
\[
U=Z\cap\left\{  i_{1},\ldots,i_{r}\right\}  \neq\varnothing\text{ and
}W=\left\{  i_{1},\ldots,i_{r}\right\}  \backslash Z\neq\varnothing.
\]
To finish the proof we shall construct a vector $\mathbf{y}\in\mathbb{S}%
_{p,+}^{n-1}$ such that $P_{A}\left(  \mathbf{y}\right)  >P_{A}\left(
\mathbf{x}\right)  =\lambda^{\left(  p\right)  }\left(  G\right)  ,$ which
yields the desired contradiction. Let $k\in W,$ and for every sufficiently
small $\varepsilon>0,$ define $\delta:=\delta\left(  \varepsilon\right)  $ by%
\[
\delta\left(  \varepsilon\right)  :=x_{k}-\sqrt[p]{x_{k}^{p}-\left\vert
U\right\vert \varepsilon^{p}}.
\]
Clearly,%
\begin{equation}
\left\vert U\right\vert \varepsilon^{p}+\left(  x_{k}-\delta\right)
^{p}=x_{k}^{p}, \label{cond1}%
\end{equation}
and $\delta\left(  \varepsilon\right)  \rightarrow0$ as $\varepsilon
\rightarrow0.$ Since $x_{j}>0$ for each $j\in W,$ we may and shall assume
that
\begin{equation}
\delta<\min_{j\in W}\left\{  x_{j}\right\}  /2\text{ \ \ \ and \ \ \ \ }%
\varepsilon<\min_{j\in W}\left\{  x_{j}\right\}  -\delta. \label{cond2}%
\end{equation}
Now, define the vector $\mathbf{y}:=\left(  y_{1},\ldots,y_{n}\right)  $ by
\[
y_{i}:=\left\{
\begin{array}
[c]{ll}%
x_{i}+\varepsilon, & \text{if }i\in U;\\
x_{k}-\delta, & \text{if }i=k;\text{ }\\
x_{i}, & \text{if }i\notin U\cup\left\{  k\right\}  .
\end{array}
\right.
\]
First, (\ref{cond1}) and (\ref{cond2}) imply that $\left\vert \mathbf{y}%
\right\vert _{p}=\left\vert \mathbf{x}\right\vert _{p}=1$ and $\mathbf{y}%
\geq0.$ Also, Bernoulli's inequality implies that $x_{k}^{p}-\left(
x_{k}-\delta\right)  ^{p}>p\delta\left(  x_{k}-\delta\right)  ^{p-1}$, and
so,
\[
r\varepsilon^{p}>\left\vert U\right\vert \varepsilon^{p}=x_{k}^{p}-\left(
x_{k}-\delta\right)  ^{p}>p\delta\left(  x_{k}-\delta\right)  ^{p-1}%
>p\delta\left(  \frac{x_{k}}{2}\right)  ^{p-1},
\]
which yields%
\begin{equation}
\delta<\frac{r}{p}\left(  \frac{2}{x_{k}}\right)  ^{p-1}\varepsilon^{p}.
\label{ud}%
\end{equation}
Further, referring to (\ref{mid}), set for short
\[
D:=r\sum_{k,j_{2},\ldots,j_{r}}a_{k,j_{2},\ldots,j_{r}}x_{j_{2}}\cdots
x_{j_{r}},
\]
and note that
\[
a_{j_{1},\ldots,j_{r}}x_{j_{1}}\cdots x_{j_{r}}=a_{j_{1},\ldots,j_{r}}%
y_{j_{1}}\cdots y_{j_{r}}%
\]
whenever $k\notin\left\{  j_{1},\ldots,j_{r}\right\}  .$ Hence
\begin{align*}
P_{A}\left(  \mathbf{y}\right)  -P_{A}\left(  \mathbf{x}\right)   &
=\sum_{k\in\left\{  j_{1},\ldots,j_{r}\right\}  }a_{j_{1},\ldots,j_{r}}\left(
y_{j_{1}}\cdots y_{j_{r}}-x_{j_{1}}\cdots x_{j_{r}}\right) \\
&  \geq a_{k,i_{2},\ldots,i_{r}}y_{k}y_{i_{2}}\cdots y_{j_{r}}-r\delta
\sum_{k,j_{2},\ldots,j_{r}}a_{k,j_{2},\ldots,j_{r}}x_{j_{2}}\cdots x_{j_{r}}.
\end{align*}
On the other hand,
\[
y_{k}y_{i_{2}}\cdots y_{j_{r}}\geq\left(  x_{k}-\delta\right)  \varepsilon
^{r-1}\geq\left(  \frac{x_{k}}{2}\right)  \varepsilon^{r-1},\text{ }%
\]
and, taking into account (\ref{ud}), we get
\begin{align*}
P_{A}\left(  \mathbf{y}\right)  -P_{A}\left(  \mathbf{x}\right)   &  \geq
a_{k,i_{2},\ldots,i_{r}}\left(  \frac{x_{k}}{2}\right)  \varepsilon
^{r-1}-\frac{r^{2}}{p}\left(  \frac{2}{x_{k}}\right)  ^{p-1}D\varepsilon^{p}\\
&  =\left(  a_{k,i_{2},\ldots,i_{r}}\left(  \frac{x_{k}}{2}\right)
-\frac{r^{2}}{p}\left(  \frac{2}{x_{k}}\right)  ^{p-1}D\varepsilon
^{p-r+1}\right)  \varepsilon^{r-1}.
\end{align*}
In view of $p-r+1>0,$ if $\varepsilon$ is sufficiently small, then
$P_{A}\left(  \mathbf{y}\right)  -P_{A}\left(  \mathbf{x}\right)  >0,$
contradicting the inequality $P_{A}\left(  \mathbf{y}\right)  \leq
P_{A}\left(  \mathbf{x}\right)  ,$ and completing the proof.
\end{proof}

Next, we prove another somewhat surprising fact, which asserts that if $p>r,$
then $\eta^{\left(  p\right)  }\left(  A\right)  $ of a symmetric nonnegative
$r$-matrix $A$ depends on all nonzero components of $A$:

\begin{theorem}
\label{PF1}Let $p>r\geq2$ and let $A$ be a symmetric nonnegative $r$-matrix.
If $A_{1},\ldots,A_{k}$ are the nonzero components of $A,$ then%
\begin{equation}
\lambda^{\left(  p\right)  }\left(  A\right)  =\left(  \sum_{i\in\left[
k\right]  }\left(  \lambda^{\left(  p\right)  }\left(  A_{i}\right)  \right)
^{p/\left(  p-r\right)  }\right)  ^{\left(  p-r\right)  /p}. \label{in2}%
\end{equation}

\end{theorem}

\begin{proof}
Clearly, we may assume that $A$ has no zero slices. Suppose that $N_{1}%
,\ldots,N_{k}$ \ are the index sets of $A_{1},\ldots,A_{k},$ let $\mathbf{x}$
be a nonnegative eigenvector to $\lambda^{\left(  p\right)  }\left(  A\right)
,$ and set $\mathbf{x}_{1}:=\mathbf{x}|_{N_{1}},\ldots,\mathbf{x}%
_{k}:=\mathbf{x}|_{N_{k}}$. Clearly,%
\[
\lambda^{\left(  p\right)  }\left(  A\right)  =P_{A}\left(  \mathbf{x}\right)
=\sum_{i\in\left[  k\right]  }P_{A_{i}}\left(  \mathbf{x}_{i}\right)  \leq
\sum_{i\in\left[  k\right]  }\lambda^{\left(  p\right)  }\left(  A_{i}\right)
\left\vert \mathbf{x}_{i}\right\vert _{p}^{r}.
\]
Letting $s=p/r$ and $t=p/\left(  p-r\right)  ,$ we see that
\[
1/s+1/t=r/p+\left(  p-r\right)  /p=1.
\]
Now, H\"{o}lder's inequality implies that%
\begin{align*}
\sum_{i\in\left[  k\right]  }\lambda^{\left(  p\right)  }\left(  A_{i}\right)
\left\vert \mathbf{x}_{i}\right\vert _{p}^{r}  &  \leq\left(  \sum
_{i\in\left[  k\right]  }\left(  \lambda^{\left(  p\right)  }\left(
A_{i}\right)  \right)  ^{t}\right)  ^{1/t}\left(  \sum_{i\in\left[  k\right]
}\left(  \left\vert \mathbf{x}_{i}\right\vert _{p}^{r}\right)  ^{s}\right)
^{1/s}\\
&  =\left(  \sum_{i\in\left[  k\right]  }\left(  \lambda^{\left(  p\right)
}\left(  A_{i}\right)  \right)  ^{p/\left(  p-r\right)  }\right)  ^{\left(
p-r\right)  /p}\left(  \sum_{i\in\left[  k\right]  }\left\vert \mathbf{x}%
_{i}\right\vert _{p}^{p}\right)  ^{r/p}\\
&  =\left(  \sum_{i\in\left[  k\right]  }\left(  \lambda^{\left(  p\right)
}\left(  A_{i}\right)  \right)  ^{p/\left(  p-r\right)  }\right)  ^{\left(
p-r\right)  /p}.
\end{align*}

To finish the proof of (\ref{in2}), we need to prove the opposite inequality.
For each $i\in\left[  k\right]  ,$ choose an eigenvector $\mathbf{z}_{i}$ to
$\lambda^{\left(  p\right)  }\left(  A_{i}\right)  ;$ then scale each
$\mathbf{z}_{i}$ so that $\sum_{i=1}^{k}\left\vert \mathbf{z}_{i}\right\vert
_{p}^{p}=1$ and $\left(  \left\vert \mathbf{z}_{1}\right\vert ^{s}%
,,\ldots,\left\vert \mathbf{z}_{k}\right\vert ^{s}\right)  $ is collinear to
$((\lambda^{\left(  p\right)  }\left(  A_{1}\right)  )^{t},\ldots
,(\lambda^{\left(  p\right)  }\left(  A_{k}\right)  )^{t}).$ Now, define a
vector $\mathbf{u}$ piecewise, by letting $\mathbf{u}$ be equal to
$\mathbf{z}_{i}$ within $N_{i}$ for each $i=1,\ldots,k$ . We see that
$\left\vert \mathbf{u}\right\vert _{p}=1$ and
\begin{align*}
\lambda^{\left(  p\right)  }\left(  A\right)   &  \geq P_{A}\left(
\mathbf{u}\right)  =\left(  \sum_{i\in\left[  k\right]  }\left(
\lambda^{\left(  p\right)  }\left(  A_{i}\right)  \right)  ^{t}\right)
^{1/t}\left(  \sum_{i\in\left[  k\right]  }\left(  \left\vert \mathbf{z}%
_{i}\right\vert _{p}^{r}\right)  ^{s}\right)  ^{1/s}\\
&  =\left(  \sum_{i\in\left[  k\right]  }\left(  \lambda^{\left(  p\right)
}\left(  A_{i}\right)  \right)  ^{p/\left(  p-r\right)  }\right)  ^{\left(
p-r\right)  /p},
\end{align*}
completing the proof of (\ref{in2}).
\end{proof}

Clearly, Theorems \ref{PF0} and \ref{PF1} imply the following simple corollary:

\begin{corollary}
\label{cor}Let $p>r\geq2$ and let $A$ be a symmetric nonnegative $r$-matrix
with no zero slices. If $\mathbf{x}$ is a nonnegative eigenvector to
$\lambda^{\left(  p\right)  }\left(  A\right)  ,$ then $\mathbf{x}$ is positive.
\end{corollary}

Next, we show that if $A$ is a symmetric nonnegative $r$-matrix and $p\geq r,$
then $\lambda^{\left(  p\right)  }\left(  G\right)  $ is the only eigenvalue
with a positive eigenvector. This fact is known for $p=r$ (see, e.g.,
\cite{YaYa11}).

\begin{proposition}
\label{PF2}Let $p\geq r\geq2,$ and let $A$ be a symmetric nonnegative
$r$-matrix of order $n$. If $\mathbf{x}:=\left(  x_{1},\ldots,x_{n}\right)  $
is a positive vector with $\left\vert \mathbf{x}\right\vert _{p}=1$,
satisfying the equations%
\[
\lambda x_{k}^{p-1}=\frac{1}{r}\frac{\partial P_{A}\left(  \mathbf{x}\right)
}{\partial x_{k}},\text{ \ \ \ }k=1,\ldots,n
\]
for some real $\lambda,$ then $\lambda=\lambda^{\left(  p\right)  }\left(
G\right)  .$
\end{proposition}

\begin{proof}
Let $\mathbf{y}:=\left(  y_{1},\ldots,y_{n}\right)  $ be a nonnegative
eigenvector to $\lambda^{\left(  p\right)  }\left(  A\right)  $ and let
\[
\sigma:=\min\left\{  x_{i}/y_{i}:y_{i}>0\right\}  .
\]
Clearly $\sigma>0$ and also $\sigma\leq1,$ for otherwise $\left\vert
\mathbf{x}\right\vert _{p}>\left\vert \mathbf{y}\right\vert _{p},$ a
contradiction. Note that $x_{i}\geq\sigma y_{i}$ for every $i\in\left[
n\right]  .$ Since $x_{k}=\sigma y_{k}$ for some $k\in\left[  n\right]  ,$ we
see that
\[
\lambda x_{k}^{p-1}\geq\frac{1}{r}\frac{\partial P_{A}\left(  \mathbf{x}%
\right)  }{\partial x_{k}}\geq\frac{1}{r}\sigma^{r-1}\frac{\partial
P_{A}\left(  \mathbf{y}\right)  }{\partial y_{k}}=\sigma^{r-1}\lambda^{\left(
p\right)  }\left(  A\right)  y_{k}^{p-1}=\sigma^{r-p}\lambda^{\left(
p\right)  }\left(  A\right)  x_{k}^{p-1}.
\]
implying that $\lambda^{\left(  p\right)  }\left(  A\right)  \leq\lambda.$ On
the other hand, Proposition \ref{pro} implies that $\lambda\leq\lambda
^{\left(  p\right)  }\left(  A\right)  ,$ and so $\lambda=\lambda^{\left(
p\right)  }\left(  A\right)  .$
\end{proof}

For the proof of Theorem \ref{th3} we also need another well-known fact, which
is proved here for completeness. Note that it is valid for any cubical
nonnegative $r$-matrix, not necessarily symmetric.

\begin{proposition}
\label{pro4}Let $A$ be a nonnegative cubical matrix of order $n.$ If $\left(
x_{1},\ldots,x_{n}\right)  $ is a nonnegative nonzero vector, then there is a
$k\in\left[  n\right]  $ such that%
\begin{equation}
\rho\left(  A\right)  x_{k}^{r-1}\leq\sum_{i_{2},\ldots,i_{r}}a_{k,i_{2}%
,\ldots,i_{r}}x_{i_{2}}\cdots x_{i_{r}}. \label{ineqa}%
\end{equation}

\end{proposition}

\begin{proof}
The assertion is obvious if $\left(  x_{1},\ldots,x_{n}\right)  $ has zero
entries. Thus, let \ $\mathbf{x}:=\left(  x_{1},\ldots,x_{n}\right)  $ be a
positive vector. The Perron-Frobenius theory developed in \cite{CPZ08}%
,\cite{FGH11}, and \cite{YaYa11} implies that $\rho\left(  A\right)  $ is an
eigenvalue of $A,$ and it has a nonnegative eigenvector $\left(  y_{1}%
,\ldots,y_{n}\right)  $. Set
\[
c:=\max\left\{  y_{i}/x_{i}:i\in\left[  n\right]  \right\}
\]
and let $c=y_{k}/x_{k}$. Then%
\[
c^{r-1}\rho^{\left(  r\right)  }\left(  A\right)  x_{k}^{r-1}=\rho^{\left(
r\right)  }\left(  A\right)  y_{k}^{r-1}=\sum_{i_{2},\ldots,i_{r}}%
a_{k,i_{2},\ldots,i_{r}}y_{i_{2}}\cdots y_{i_{r}}\leq c^{r-1}\sum
_{i_{2},\ldots,i_{r}}a_{k,i_{2},\ldots,i_{r}}x_{i_{2}}\cdots x_{i_{r}},
\]
implying (\ref{ineqa}).
\end{proof}

\subsection{\label{rpar}$r$-partite $r$-matrices}

Bipartite graphs are fundamental building blocks in structural theorems for
$2$-graphs, like, e.g., in Szemer\'{e}di's Regularity Lemma. For $r$-uniform
hypergraphs a similar role is played by the $r$-partite $r$-graphs. This
concept can be extended to matrices, so in this section we define $r$-partite
$r$-matrices and prepare the introduction of the symmetrant of a matrix. Both
these concepts are based on partitions of the index set of cubical matrices.

Given a partition $\left[  n\right]  =N_{1}\cup\cdots\cup N_{r}$,\ define the
functions $\eta:\left[  n\right]  \rightarrow\left[  r\right]  $ and $\theta:$
$\left[  n\right]  \rightarrow$ $\left[  n\right]  ,$ called \emph{selector}
and \emph{locator} of the partition, as follows: for any $x\in\left[
n\right]  ,$ let $\eta\left(  x\right)  $ be the unique number such that $x\in
N_{\eta\left(  x\right)  },$ and let $\theta\left(  x\right)  $ be the
relative position of $x$ in the set $N_{\eta\left(  x\right)  },$ in the
ordering induced by $\left[  n\right]  .$

A cubical $r$-matrix $A$ of order $n$ is called\emph{ }$r$\emph{-partite} if
there is a partition $\left[  n\right]  =N_{1}\cup\cdots\cup N_{r}$ with
selector $\eta\left(  x\right)  $ such that if $a_{i_{1},\ldots,i_{r}}\neq0,$
then $\eta\left(  i_{1}\right)  ,\ldots,\eta\left(  i_{r}\right)  $ are
distinct. E.g., a square $2$-matrix is bipartite if after a permutation of its
index set, it can be written as a block matrix
\[
\left[
\begin{array}
[c]{cc}%
0 & A_{2}\\
A_{1} & 0
\end{array}
\right]  ,
\]
where the zero diagonal blocks are square.

Here is a crucial theorem, which seems new even for bipartite $2$-matrices.
For bipartite graphs and $p=2$ it was proved in \cite{BFP08} by another method.

\begin{theorem}
\label{th1}Let $p\geq1$, let $A$ be an $r$-partite symmetric $r$-matrix of
order $n,$ and let $\left[  n\right]  =N_{1}\cup\cdots\cup N_{r}$ be its
partition. If $\mathbf{x}$ is an eigenvector to $\eta^{\left(  p\right)
}\left(  A\right)  ,$ then for every $i\in\left[  r\right]  ,$ the vector
$\mathbf{x}|_{N_{i}}$ satisfies
\[
\left\vert \mathbf{x}|_{N_{i}}\right\vert _{p}=r^{-1/p}.
\]

\end{theorem}

\begin{proof}
Let $\mathbf{x}^{\left(  1\right)  }:=\mathbf{x}|_{N_{1}},\ldots
,\mathbf{x}^{\left(  r\right)  }:=\mathbf{x}|_{N_{r}},$ and note that
\[
|\mathbf{x}^{\left(  1\right)  }|_{p}^{p}+\cdots+|\mathbf{x}^{\left(
r\right)  }|_{p}^{p}=\left\vert \mathbf{x}\right\vert _{p}^{p}=1.
\]
By symmetry, suppose that $|\mathbf{x}^{\left(  1\right)  }|_{p}\leq\cdots
\leq|\mathbf{x}^{\left(  r\right)  }|_{p}.$ Clearly, $|\mathbf{x}^{\left(
1\right)  }|_{p}>0,$ for otherwise $\mathbf{x}^{\left(  1\right)  }=0,$ and so
$P_{A}\left(  \mathbf{x}\right)  =0.$

Assume that the conclusion of the theorem fails for some $i\in\left[
r\right]  ,$ which obviously implies that $|\mathbf{x}^{\left(  1\right)
}|_{p}<|\mathbf{x}^{\left(  r\right)  }|_{p}.$ Let
\[
\alpha:=\sqrt{|\mathbf{x}^{\left(  r\right)  }|_{p}/|\mathbf{x}^{\left(
1\right)  }|_{p}},\text{ \ \ and \ \ \ }\beta:=\sqrt{|\mathbf{x}^{\left(
1\right)  }|_{p}/|\mathbf{x}^{\left(  r\right)  }|_{p}},
\]
and define an $n$-vector $\mathbf{y}$ by setting
\[
\mathbf{y}|_{N_{1}}:=\alpha\mathbf{x}^{\left(  1\right)  },\text{
\ \ }\mathbf{y}|_{N_{r}}:=\beta\mathbf{x}^{\left(  r\right)  }%
\]
and letting $\mathbf{y}$ be the same as $\mathbf{x}$ over the set $N_{2}%
\cup\cdots\cup N_{r-1}.$ First, note that $\left\vert \mathbf{y}\right\vert
_{p}<1.$ Indeed,%
\begin{align*}
\left\vert \mathbf{y}\right\vert _{p}^{p}  &  =\alpha^{p}|\mathbf{x}^{\left(
1\right)  }|_{p}^{p}+|\mathbf{x}^{\left(  2\right)  }|_{p}^{p}+\cdots
+|\mathbf{x}^{\left(  r-1\right)  }|_{p}^{p}+\beta^{p}|\mathbf{x}^{\left(
r\right)  }|_{p}^{p}\\
&  =2\sqrt{|\mathbf{x}^{\left(  1\right)  }|_{p}^{p}|\mathbf{x}^{\left(
r\right)  }|_{p}^{p}}+|\mathbf{x}^{\left(  2\right)  }|_{p}^{p}+\cdots
+|\mathbf{x}^{\left(  r-1\right)  }|_{p}^{p}\\
&  <|\mathbf{x}^{\left(  1\right)  }|_{p}^{p}+\cdots+|\mathbf{x}^{\left(
r\right)  }|_{p}^{p}=1.
\end{align*}
On the other hand, $\alpha\beta=1,$ and so,%
\[
P_{A}\left(  \mathbf{y}\right)  =\sum_{i_{1},\ldots,i_{r}}a_{i_{1}%
,\ldots,i_{r}}y_{i_{1}}\cdots y_{i_{r}}=\sum_{i_{1},\ldots,i_{r}}%
a_{i_{1},\ldots,i_{r}}\alpha x_{i_{1}}\cdots\beta x_{i_{r}}=P_{A}\left(
\mathbf{x}\right)  \text{.}%
\]
Hence, $\mathbf{y}$ is an eigenvector to $\rho^{\left(  p\right)  }\left(
A\right)  .$ However, $||\mathbf{y}|_{p}^{-1}\mathbf{y}|_{p}=1,$ and so
\[
P_{A}\left(  \mathbf{x}\right)  =P_{A}\left(  \mathbf{y}\right)
<|\mathbf{y}|_{p}^{-r}P_{A}\left(  \mathbf{y}\right)  =P_{A}(|\mathbf{y}%
|_{p}^{-1}\mathbf{y})\leq P_{A}\left(  \mathbf{x}\right)  .
\]
This contradiction completes the proof.
\end{proof}

\subsection{\label{symas}The symmetrant of a matrix}

In this section we discuss the symmetrant of a matrix, a concept that has been
introduced by Ragnarsson and Van Loan in \cite{RaVL13}.

Suppose that $A$ is a real $r$-matrix of order $n_{1}\times\cdots\times
n_{r}.$ Set $n:=n_{1}+\cdots+n_{r},$ partition $\left[  n\right]  $ into $r$
consecutive intervals $N_{1},\ldots,N_{r}$ with $\left\vert N_{1}\right\vert
=n_{1},\ldots,\left\vert N_{r}\right\vert =n_{r},$ and let $\eta\left(
x\right)  $ and $\theta\left(  x\right)  $ be the selector and the locator of
this partition. Now, define an $r$-matrix $B$ of order $n$ by%
\begin{equation}
b_{j_{1},\ldots,j_{r}}:=\left\{
\begin{array}
[c]{ll}%
0, & \text{if }\eta\left(  j_{1}\right)  ,\ldots,\eta\left(  j_{r}\right)
\text{ are not all distinct;}\\
a_{i_{1},\ldots,i_{r}}, & \text{(}i_{\eta\left(  j_{s}\right)  }%
:=\theta\left(  j_{s}\right)  \text{ for all }s\in\left[  r\right]  \text{),
otherwise.}%
\end{array}
\right.  \label{symr}%
\end{equation}
The matrix $B$ will be called the \emph{symmetrant} of $A$ and will be denoted
by \textrm{sym}$\left(  A\right)  $.

The correspondence $(j_{1},\ldots,j_{r})\rightarrow(i_{1},\ldots,i_{r})$ in
(\ref{symr}) can be described also as follows: if $j_{1},\ldots,j_{r}$ belong
to different sets $N_{1},\ldots,N_{r}$, then reorder $j_{1},\ldots,j_{r}$ into
$j_{1}^{\prime},\ldots,j_{r}^{\prime}$ so that $\eta\left(  j_{s}^{\prime
}\right)  $ increases with $s,$ and let $i_{1}:=\theta\left(  j_{1}^{\prime
}\right)  ,\ldots,i_{r}:=\theta\left(  j_{r}^{\prime}\right)  .$

Notwithstanding its intricate definition, the symmetrant is quite natural: the
partition $\left[  n\right]  =N_{1}\cup\cdots\cup N_{r}$ splits \textrm{sym}%
$\left(  A\right)  $ into $r^{r}$ blocks, of which only $r!$ blocks are
nonzero---the $r!$ transposes of $A$. Thus, each nonzero block is induced by a
particular choice of $N_{i_{1}}\subset\left[  n\right]  ,\ldots,N_{i_{r}%
}\subset\left[  n\right]  $ such that $i_{1},\ldots,i_{r}$ is a permutation of
$1,\ldots,r$. The case $r=2$ is visualized in (\ref{sym2}) and can be written
also as $B=$ \textrm{sym}$\left(  A\right)  $.

Let us stress that (\ref{symr}) immediately implies that \textrm{sym}$\left(
A\right)  $ is symmetric and $r$-partite.

The following theorem is crucial for the proofs of Theorems \ref{mth} and
\ref{mth1}:\ 

\begin{theorem}
\label{th2} If $p>1$ and $A$ is a real $r$-matrix, then the following
relations hold:

(a)
\begin{equation}
\eta^{\left(  p\right)  }\left(  \mathrm{sym}\left(  A\right)  \right)
\leq\frac{r!}{r^{r/p}}\left\Vert A\right\Vert _{p}\text{;} \label{basin}%
\end{equation}

(b) if $A$ is nonnegative, then
\[
\eta^{\left(  p\right)  }\left(  \mathrm{sym}\left(  A\right)  \right)
=\frac{r!}{r^{r/p}}\left\Vert A\right\Vert _{p}\text{.}%
\]

\end{theorem}

\begin{proof}
\emph{(a)} Suppose that $A$ is a matrix of size $n_{1}\times\cdots\times
n_{r}$, set $n:=n_{1}+\cdots+n_{r},$ partition $\left[  n\right]  $ into $r$
consecutive intervals $N_{1},\ldots,N_{r}$ with $\left\vert N_{1}\right\vert
=n_{1},\ldots,\left\vert N_{r}\right\vert =n_{r},$ and let $\eta\left(
x\right)  $ and $\theta\left(  x\right)  $ be the selector and the locator of
this partition.

Let $\mathbf{x}$ be a real $n$-vector and let $\mathbf{x}^{\left(  1\right)
}:=\mathbf{x}|_{N_{1}},\ldots,\mathbf{x}^{\left(  r\right)  }:=\mathbf{x}%
|_{N_{r}}.$ It is not hard to check the identity
\[
P_{\mathrm{sym}\left(  A\right)  }\left(  \mathbf{x}\right)  =r!L_{A}%
(\mathbf{x}^{\left(  1\right)  },\ldots,\mathbf{x}^{\left(  r\right)  }).
\]
Indeed each nonzero block of $\mathrm{sym}\left(  A\right)  $ is induced by a
particular choice of $N_{i_{1}}\subset\left[  n\right]  ,\ldots,N_{i_{r}%
}\subset\left[  n\right]  $ such that $i_{1},\ldots,i_{r}$ is a permutation of
$1,\ldots,r$; thus, denote such a block\ by $B_{i_{1},\ldots,i_{r}}.$ Note
that $B_{i_{1},\ldots,i_{r}}$ is a transpose of $A$, and
\[
L_{B_{i_{1},\ldots,i_{r}}}(\mathbf{x}^{\left(  i_{1}\right)  },\ldots
,\mathbf{x}^{\left(  i_{r}\right)  })=L_{A}(\mathbf{x}^{\left(  1\right)
},\ldots,\mathbf{x}^{r}).
\]
Now we see that
\begin{align*}
P_{\mathrm{sym}\left(  A\right)  }\left(  \mathbf{x}\right)   &
=\sum\{L_{B_{i_{1},\ldots,i_{r}}}(\mathbf{x}^{\left(  i_{1}\right)  }%
,\ldots,\mathbf{x}^{\left(  i_{r}\right)  }):i_{1},\ldots,i_{r}\text{ is a
permutation of }1,\ldots,r\}\\
&  =r!L_{A}(\mathbf{x}^{\left(  1\right)  },\ldots,\mathbf{x}^{r})
\end{align*}
as claimed.

Hence, if $\mathbf{x}$ is an eigenvector to $\eta^{\left(  p\right)  }\left(
\mathrm{sym}\left(  A\right)  \right)  $, Theorem \ref{th1} implies that%
\[
|\mathbf{x}^{\left(  1\right)  }|_{p}=\cdots=|\mathbf{x}^{\left(  r\right)
}|_{p}=r^{-1/p},
\]
and therefore,
\[
\eta^{\left(  p\right)  }\left(  \mathrm{sym}\left(  A\right)  \right)
=|P_{\mathrm{sym}\left(  A\right)  }\left(  \mathbf{x}\right)  |~=r!|L_{A}%
(\mathbf{x}^{\left(  1\right)  },\ldots,\mathbf{x}^{\left(  r\right)  }%
)|~\leq\frac{r!}{r^{r/p}}\left\Vert A\right\Vert _{p},
\]
proving (\ref{basin}).

\emph{(b)} Suppose that $A$ is nonnegative and let $\mathbf{x}^{\left(
1\right)  },\ldots,\mathbf{x}^{\left(  r\right)  }$ be an eigenkit to
$\left\Vert A\right\Vert _{p}$. Note that in general $\mathbf{x}^{\left(
1\right)  },\ldots,\mathbf{x}^{\left(  r\right)  }$ may be complex vectors,
but we suppose that they are nonnegative, because
\[
\left\Vert A\right\Vert _{p}=|L_{A}(\mathbf{x}^{\left(  1\right)  }%
,\ldots,\mathbf{x}^{\left(  r\right)  })|\text{~}\leq L_{A}(|\mathbf{x}%
^{\left(  1\right)  }|,\ldots,|\mathbf{x}^{\left(  r\right)  }|)\leq\left\Vert
A\right\Vert _{p}.
\]
Lagrange's method implies that for any $k\in\left[  r\right]  ,$ there exists
a $\mu_{k}$ such that for every $s\in\left[  n_{k}\right]  $,%
\begin{align}
\mu_{k}(x_{s}^{\left(  k\right)  })^{p-1}  &  =\frac{\partial L_{A}%
(\mathbf{x}^{\left(  1\right)  },\ldots,\mathbf{x}^{\left(  r\right)  }%
)}{\partial x_{j}^{\left(  k\right)  }}\label{e}\\
&  =\sum\{a_{i_{1},\ldots,i_{r}}x_{i_{1}}^{\left(  1\right)  }\cdots
x_{i_{k-1}}^{\left(  k-1\right)  }x_{i_{k+1}}^{\left(  k+1\right)  }\cdots
x_{i_{r}}^{\left(  r\right)  }:\text{ }i_{k}=s,\text{ }i_{j}\in\left[
n_{j}\right]  \text{ for }j\neq k\}.\nonumber
\end{align}
Multiplying this equation by $x_{s}^{\left(  k\right)  },$ and adding all
equations for $s\in\left[  n_{k}\right]  $, we see that
\[
\mu_{k}=\mu_{k}\sum_{s\in\left[  n_{k}\right]  }|x_{s}^{\left(  k\right)
}|_{p}^{p}=L_{A}(\mathbf{x}^{\left(  1\right)  },\ldots,\mathbf{x}^{\left(
r\right)  })=\left\Vert A\right\Vert _{p}.
\]
Hence $\mu_{k}=\left\Vert A\right\Vert _{p}$ for every $k\in\left[  r\right]
$.

Next, write $\mathbf{x}$ for the $n$-vector defined piecewise by
$\mathbf{x}|_{N_{1}}:=\mathbf{x}^{\left(  1\right)  },\ldots,\mathbf{x}%
|_{N_{r}}:=\mathbf{x}^{\left(  r\right)  }$. Let $\mathbf{y}:=r^{-1/p}%
\mathbf{x}$, and note that $\left\vert \mathbf{y}\right\vert _{p}=1$.

Suppose that $i\in\left[  n\right]  $, and set $k:=\eta\left(  i\right)  $
and\ $s:=\theta\left(  i\right)  $. Clearly,
\[
\left(  r-1\right)  !\frac{\partial L_{A}(\mathbf{x}^{\left(  1\right)
},\ldots,\mathbf{x}^{\left(  r\right)  })}{\partial x_{s}^{\left(  k\right)
}}=\frac{1}{r}\frac{\partial P_{\mathrm{sym}\left(  A\right)  }\left(
\mathbf{x}\right)  }{\partial x_{i}},
\]
and (\ref{e}) implies that
\[
\left\Vert A\right\Vert _{p}r^{\left(  p-1\right)  /p}y_{i}^{p-1}=\left\Vert
A\right\Vert _{p}x_{i}^{p-1}=\frac{1}{r!}\frac{\partial P_{\mathrm{sym}\left(
A\right)  }\left(  \mathbf{x}\right)  }{\partial x_{i}}=\frac{1}{r!}%
\frac{\partial P_{\mathrm{sym}\left(  A\right)  }\left(  \mathbf{y}\right)
}{\partial y_{i}}r^{\left(  r-1\right)  /p}.
\]
Hence, for every $i\in\left[  n\right]  $,%
\[
\frac{r!}{r^{r/p}}\left\Vert A\right\Vert _{p}y_{i}^{p-1}=\frac{1}{r}%
\frac{\partial P_{\mathrm{sym}\left(  A\right)  }\left(  \mathbf{y}\right)
}{\partial y_{i}}.
\]
Therefore, Proposition \ref{pro} implies that
\[
\eta^{\left(  p\right)  }\left(  \mathrm{sym}\left(  A\right)  \right)
\geq\frac{r!}{r^{r/p}}\left\Vert A\right\Vert _{p},
\]
completing the proof of Theorem \ref{th2}.
\end{proof}

Armed with Theorem \ref{th2} and the results of Section \ref{PFs}, we
encounter no difficulty in proving Theorem \ref{mth1}.\medskip

\begin{proof}
[\textbf{Proof of Theorem \ref{mth1}}]Suppose that $A$ is symmetric and
nonnegative $r$-matrix of order $n.$ First, we prove the assertion for $p>r$,
and then obtain the case $p=r$ by passing to limit. Thus, suppose that $p>r$
and let $\mathbf{x}$ be an eigenvector to $\lambda^{\left(  p\right)  }\left(
A\right)  ,$ which by Corollary \ref{cor} is positive. Let $n^{\prime}=rn$,
and suppose that $\left[  n^{\prime}\right]  =N_{1}\cup\cdots\cup N_{r}$ is
the partition of $\mathrm{sym}\left(  A\right)  $. Write $\mathbf{y}$ for the
$n$-vector defined piecewise by $\mathbf{y}|_{N_{1}}:=\mathbf{x}%
,\ldots,\mathbf{y}|_{N_{r}}:=\mathbf{x}$. Let $\mathbf{z}:=r^{-1/p}\mathbf{y}%
$, and note that $\left\vert \mathbf{z}\right\vert _{p}=1.$ Following the
argument of clause \emph{(b)} of Theorem \ref{th2}, we conclude that\emph{ \ }%
\[
\eta^{\left(  p\right)  }\left(  \mathrm{sym}\left(  A\right)  \right)
=\frac{r!}{r^{r/p}}\lambda^{\left(  p\right)  }\left(  A\right)  ,
\]
and therefore,
\[
\eta^{\left(  p\right)  }\left(  A\right)  =\lambda^{\left(  p\right)
}\left(  A\right)  =\left\Vert A\right\Vert _{p},
\]
as claimed.

On the other hand, Propositions \ref{pro2} and \ref{pro3} imply that
$\lambda^{\left(  p\right)  }\left(  A\right)  $ and $\left\Vert A\right\Vert
_{p}$ are continuous in $p,$ so letting $p\rightarrow r,$ we see that
\[
\eta\left(  A\right)  =\lambda\left(  A\right)  =\left\Vert A\right\Vert
_{r},
\]
completing the proof of Theorem \ref{mth1}.
\end{proof}

\section{\label{boms}Bounds on the spectral $p$-norm of matrices}

In this section we use results from the previous sections\ to prove various
bounds on the spectral norms of $r$-matrices. Motivated by a well-known and
very usable bound for $2$-graphs, in Section \ref{prf} we give an upper bound
on $\left\Vert A\right\Vert _{p},$ which helps to conclude the proof of
Theorem \ref{mth}, but is also of independent interest. In Section \ref{ulb},
we give a few general bounds on $\left\Vert A\right\Vert _{p}$, in particular,
bounds related to regular matrices.

\subsection{\label{prf}An upper bound on $\left\Vert A\right\Vert _{r}$ and a
proof of Theorem \ref{mth}}

The main purpose of this section is to prove a combinatorial bound on
$\left\Vert A\right\Vert _{r},$ and apply this bound to prove Theorem
\ref{mth}.

Suppose that $A$ is an $m\times n$ nonnegative $2$-matrix. Recall that
$\left\Vert A\right\Vert _{2}^{2}$ is the largest eigenvalue of $AA^{T}$ and
$A^{T}A$; hence $\left\Vert A\right\Vert _{2}^{2}$ does not exceed either of
the maximum rowsums of $AA^{T}$ and $A^{T}A$, i.e.,
\[
\left\Vert A\right\Vert _{2}^{2}\leq\min\left\{  \max_{s\in\left[  m\right]
}\sum_{j\in\left[  n\right]  }a_{s,j}\sum_{k\in\left[  m\right]  }a_{k,j}%
,\max_{t\in\left[  n\right]  }\sum_{j\in\left[  m\right]  }a_{j,t}\sum
_{k\in\left[  n\right]  }a_{j,k}\right\}
\]

In the next theorem we generalize this bound to $r$-matrices:

\begin{theorem}
\label{th3}If $A$ is an $r$-matrix of order $n_{1}\times\cdots\times n_{r}$,
then%
\begin{equation}
\left\Vert A\right\Vert _{r}^{r}\leq\min_{k\in\left[  r\right]  }\left\{
\max_{s\in\left[  n_{k}\right]  }\left\{  \sum\{|a_{i_{1},\ldots,i_{r}}%
|\prod_{j\in\left[  r\right]  \backslash\left\{  k\right\}  }|A_{i_{j}%
}^{\left(  k\right)  }|_{1}:i_{k}=s,\text{ }i_{j}\in\left[  n_{j}\right]
\text{ for }j\neq k\text{ }\}\right\}  \right\}  \label{in1}%
\end{equation}

\end{theorem}

\begin{proof}
Since $\left\Vert A\right\Vert _{r}\leq\left\Vert \left\vert A\right\vert
\right\Vert _{r},$ without loss of generality, we assume that $A$ is
nonnegative. Thus, letting $B:=\mathrm{sym}\left(  A\right)  ,$ Theorem
\ref{th2} implies that
\[
\lambda\left(  B\right)  =\frac{r!}{r^{r/p}}\left\Vert A\right\Vert _{r}.
\]
For any $i\in\left[  n\right]  ,$ set
\[
d_{i}:=\frac{1}{\left(  r-1\right)  !}\sum_{j_{2},\ldots,j_{r}}b_{i,j_{2}%
,\ldots,j_{r}}.
\]
Without loss of generality we assume that $A$ has no zero slices; thus, $B$
has no zero slices either; hence, $d_{i}>0$ for every $i\in\left[  n.\right]
.$

Letting $k:=\eta\left(  i\right)  $ and $s=\theta\left(  i\right)  $, it is
not hard to see that
\[
d_{i}=\sum_{i_{k}=s}a_{i_{1},\ldots,i_{r}}=|A_{s}^{\left(  k\right)  }|_{1}.
\]
Now let $\mathbf{y}:=(d_{1}^{1/r},\ldots,d_{n}^{1/r}),$ and observe that
Proposition \ref{pro4} implies that for some $i\in\left[  n\right]  ,$
\[
\eta\left(  B\right)  \leq y_{i}^{-r+1}\sum_{j_{2},\ldots,j_{r}}%
b_{i,j_{2},\ldots,j_{r}}y_{j_{2}}\cdots y_{j_{r}}=\frac{\left(  r-1\right)
!}{y_{i}^{r-1}}\sum_{j_{2},\ldots,j_{r}}\frac{b_{i,j_{2},\ldots,j_{r}}%
}{\left(  r-1\right)  !}y_{j_{2}}\cdots y_{j_{r}}.
\]
Dividing both sides by $\left(  r-1\right)  !$ and using Theorem \ref{th2}, we
find that
\[
\left\Vert A\right\Vert _{r}=\frac{\eta\left(  B\right)  }{\left(  r-1\right)
!}\leq\frac{1}{d_{i}^{\left(  r-1\right)  /r}}\sum_{j_{2},\ldots,j_{r}}%
\frac{b_{i,j_{2},\ldots,j_{r}}}{\left(  r-1\right)  !}d_{j_{2}}^{1/r}\cdots
d_{j_{r}}^{1/r}.
\]
In view of the identity
\[
\sum_{j_{2},\ldots,j_{r}}\frac{b_{i,j_{2},\ldots,j_{r}}}{d_{i}\left(
r-1\right)  !}=1,
\]
the Power Mean Inequality implies that
\begin{equation}
\sum_{j_{2},\ldots,j_{r}}\frac{b_{i,j_{2},\ldots,j_{r}}}{d_{i}\left(
r-1\right)  !}d_{j_{2}}^{1/r}\cdots d_{j_{r}}^{1/r}\leq\left(  \sum
_{j_{2},\ldots,j_{r}}\frac{b_{i,j_{2},\ldots,j_{r}}}{d_{i}\left(  r-1\right)
!}d_{j_{2}}\cdots d_{j_{r}}\right)  ^{1/r}, \label{pmin}%
\end{equation}
and so%
\[
\left\Vert A\right\Vert _{r}^{r}\leq\sum_{j_{2},\ldots,j_{r}}\frac
{b_{i,j_{2},\ldots,j_{r}}}{\left(  r-1\right)  !}d_{j_{2}}\cdots d_{j_{r}}.
\]
Letting $k=\eta\left(  i\right)  $ and $s=\theta\left(  i\right)  ,$ it is not
hard to see that for each $j\in\left[  r\right]  \backslash\left\{  k\right\}
,$ there exists $i_{j}\in\left[  n_{s}\right]  $ such that
\[
\sum_{j_{2},\ldots,j_{r}}\frac{b_{i,j_{2},\ldots,j_{r}}}{\left(  r-1\right)
!}d_{j_{2}}\cdots d_{j_{r}}=\sum\{a_{i_{1},\ldots,i_{r}}\prod_{j\in\left[
r\right]  \backslash\left\{  k\right\}  }|A_{i_{j}}^{\left(  k\right)  }%
|_{1}:i_{k}=s,\text{ }i_{j}\in\left[  n_{j}\right]  \text{ for }j\neq k\text{
}\}.
\]
and inequality (\ref{in1}) follows.
\end{proof}

Using Theorem \ref{th3}, just a minor extra effort is needed to prove Theorem
\ref{mth}. Note that our proof extends the idea of \cite{Nik07}, which is
different from the main idea of Kolotilina \cite{Kol06}.\medskip

\begin{proof}
[\textbf{Proof of Theorem \ref{mth}}]Since $\left\Vert A\right\Vert _{r}%
\leq\left\Vert \left\vert A\right\vert \right\Vert _{r},$ without loss of
generality, we assume that $A$ is nonnegative, and so Theorem \ref{th3}
implies that%
\[
\left\Vert A\right\Vert _{r}^{r}\leq\min_{k\in\left[  r\right]  }\max
_{s\in\left[  n_{k}\right]  }\sum\{a_{i_{1},\ldots,i_{r}}\prod_{j\in\left[
r\right]  \backslash\left\{  k\right\}  }|A_{i_{j}}^{\left(  k\right)  }%
|_{1}:i_{k}=s,\text{ }i_{j}\in\left[  n_{j}\right]  \text{ for }j\neq k\}.
\]
Let the extremum in the right side be attained for $k\in\left[  r\right]  $
and $s\in\left[  n_{k}\right]  .$ Then
\begin{align*}
\left\Vert A\right\Vert _{r}^{r}  &  \leq\sum_{i_{k}=s}a_{i_{1},\ldots,i_{r}%
}\prod_{j\in\left[  r\right]  \backslash\left\{  k\right\}  }|A_{i_{j}%
}^{\left(  k\right)  }|_{1}=\sum_{i_{k}=s,\text{ }a_{i_{1},\ldots,i_{r}}%
>0}\frac{a_{i_{1},\ldots,i_{r}}}{|A_{s}^{\left(  k\right)  }|_{1}}%
|A_{s}^{\left(  k\right)  }|_{1}\prod_{j\in\left[  r\right]  \backslash
\left\{  k\right\}  }|A_{i_{j}}^{\left(  k\right)  }|_{1}\\
&  \leq\sum_{i_{k}=s,\text{ }a_{i_{1},\ldots,i_{r}}>0}\frac{a_{i_{1}%
,\ldots,i_{r}}}{|A_{s}^{\left(  k\right)  }|_{1}}\max_{a_{i_{1},\ldots,i_{r}%
}>0}|A_{i_{1}}^{\left(  1\right)  }|_{1}\cdots|A_{i_{r}}^{\left(  r\right)
}|_{1}=\max_{a_{i_{1},\ldots,i_{r}}>0}A_{l_{1}}^{\left(  1\right)  }\cdots
A_{l_{r}}^{\left(  r\right)  },
\end{align*}
completing the proof of Theorem \ref{mth}.\footnote{Note that Theorem 273 of
Hardy, Littlewood, and Polya \cite{HLP88} is very general and is widely open
for further improvements in the spirit of Theorem \ref{mth}.}
\end{proof}

\subsection{\label{ulb}A few general bounds on $\left\Vert A\right\Vert _{p}$}

In this subsection, first we study real symmetric matrices with constant slice
sums, which have some extremal spectral properties. Thus, write $\Sigma A$ for
the sum of the entries of a matrix $A.$ An $r$-matrix $A$ of order
$n_{1}\times\cdots\times n_{r}$ is called \emph{regular,} if for every
$k\in\left[  r\right]  ,$
\[
\Sigma A_{1}^{\left(  k\right)  }=\cdots=\Sigma A_{n_{k}}^{\left(  k\right)
}.
\]
Note that the adjacency matrix of a regular $2$-graph $G$ is regular, and so
is the biadjacency matrix of a semiregular bipartite $2$-graph; these facts
explain our choice for the term "regular".

As with $2$-matrices, it turns out that regularity is closely related to the
spectral radius:

\begin{proposition}
\label{preg}If $A$ is a real symmetric $r$-matrix of order $n$ and
$\eta^{\left(  p\right)  }\left(  A\right)  =n^{-r/p}|\Sigma A|$ \ for some
$p>1,$ then $A$ is regular.
\end{proposition}

\begin{proof}
Let $\lambda\in\{\lambda^{\left(  p\right)  }\left(  G\right)  ,\lambda_{\min
}^{\left(  p\right)  }\left(  G\right)  \}$ and $\eta^{\left(  p\right)
}\left(  A\right)  =\left\vert \lambda\right\vert =n^{-r/p}\left\vert \Sigma
A\right\vert =\left\vert P_{A}\left(  n^{-1/p}\mathbf{j}_{n}\right)
\right\vert .$ Then $\lambda$ satisfies the equations
\[
\lambda n^{\left(  p-1\right)  /p}=\varepsilon n^{\left(  r-1\right)  /p}%
\sum_{i_{2},\ldots,i_{r}}a_{k,i_{2},\ldots,i_{r}},\text{ \ \ \ \ }%
k=1,\ldots,n,
\]
where $\varepsilon=\pm1.$ Therefore $\Sigma A_{1}^{\left(  k\right)  }%
=\cdots=\Sigma A_{n_{k}}^{\left(  k\right)  },$ and so $A$ is regular.
\end{proof}

Moreover, if $A$ is nonnegative and $p\geq r,$ the converse of Proposition
\ref{preg} is true as well:

\begin{theorem}
\label{pro_reg}Let $p\geq r$ and $A$ be a nonnegative symmetric $r$-matrix of
order $n.$ If $A$ is regular, then $\eta^{\left(  p\right)  }\left(  A\right)
=n^{-r/p}\Sigma A.$
\end{theorem}

\begin{proof}
Recall that if \ $p\geq r$ and $A$ is nonnegative symmetric, then
$\eta^{\left(  p\right)  }\left(  A\right)  =\lambda^{\left(  p\right)
}\left(  A\right)  .$ Let $\left[  x_{i}\right]  \in$ $\mathbb{S}_{p}^{n-1}$
be a nonnegative eigenvector to $\lambda^{\left(  p\right)  }\left(  A\right)
,$ and suppose that $x_{k}=\max\left\{  x_{1},\ldots,x_{n}\right\}  .$ Since
Lagrange's method implies that
\[
\lambda^{\left(  p\right)  }\left(  A\right)  x_{k}^{p-1}=\sum_{i_{2}%
,\ldots,i_{r}}a_{k,i_{2},\ldots,i_{r}}x_{i_{2}}\cdots x_{i_{r}},
\]
and we have $x_{k}\geq n^{-1/p},$ it follows that
\[
\lambda^{\left(  p\right)  }\left(  A\right)  \leq x_{k}^{r-p}\sum
_{i_{2},\ldots,i_{r}}a_{k,i_{2},\ldots,i_{r}}\leq\left(  n^{-1/p}\right)
^{n-p}\frac{1}{n}\Sigma A=n^{-r/p}\Sigma A.
\]
In view of \
\[
\lambda^{\left(  p\right)  }\left(  A\right)  \geq P_{A}\left(  n^{-1/p}%
\mathbf{j}_{n}\right)  =n^{-r/p}\Sigma A,
\]
we conclude that $\eta^{\left(  p\right)  }\left(  A\right)  =n^{-r/p}\Sigma
A$.
\end{proof}

Next, we give bounds on $\left\Vert A\right\Vert _{p},$ which generalize
well-known facts about the $2$-spectral norm of $2$-matrices.

\begin{theorem}
\label{tub}If $p>1$ and $A$ is an $r$-matrix of order $n_{1}\times\cdots\times
n_{r}$, then
\begin{equation}
\left\Vert A\right\Vert _{p}\leq\left\vert A\right\vert _{p/\left(
p-1\right)  }. \label{upbo}%
\end{equation}
Equality holds if and only if $A$ is a rank-one matrix.
\end{theorem}

\begin{proof}
Inequality (\ref{upbo}) is straightforward, but the case of equality needs an
argument. Let \ $\mathbf{x}^{\left(  1\right)  },\ldots,\mathbf{x}^{\left(
r\right)  }$ be an eigenkit to $\left\Vert A\right\Vert _{p}.$ Thus, we have%
\begin{equation}
\left\Vert A\right\Vert _{p}=|\sum_{i_{1},\ldots,i_{r}}a_{i_{1},\ldots,i_{r}%
}\overline{x_{i_{1}}^{\left(  1\right)  }}\cdots\overline{x_{i_{r}}^{\left(
r\right)  }}|~\leq\sum_{i_{1},\ldots,i_{r}}\left\vert a_{i_{1},\ldots,i_{r}%
}\right\vert |x_{i_{1}}^{\left(  1\right)  }|\cdots|x_{i_{r}}^{\left(
r\right)  }|. \label{ub1}%
\end{equation}
Note that the number $q=p/(p-1)$ is the conjugate of $p$, since $1/p+1/q=1$;
hence, H\"{o}lder's inequality implies that%
\begin{align*}
\sum_{i_{1},\ldots,i_{r}}\left\vert a_{i_{1},\ldots,i_{r}}\right\vert
|x_{i_{1}}^{\left(  1\right)  }|\cdots|x_{i_{r}}^{\left(  r\right)  }|~  &
\leq\left(  \sum_{i_{1},\ldots,i_{r}}\left\vert a_{i_{1},\ldots,i_{r}%
}\right\vert ^{p/\left(  p-1\right)  }\right)  ^{\left(  p-1\right)
/p}\left(  \sum_{i_{1},\ldots,i_{r}}|x_{i_{1}}^{\left(  1\right)  }|^{p}%
\cdots|x_{i_{r}}^{\left(  r\right)  }|^{p}\right)  ^{1/p}\\
&  \leq\left\vert A\right\vert _{p/\left(  p-1\right)  }\left(  \sum_{i}%
|x_{i}^{\left(  1\right)  }|^{p}\right)  ^{1/p}\cdots\left(  \sum_{i}%
|x_{i}^{\left(  r\right)  }|^{p}\right)  ^{1/p}\\
&  =\left\vert A\right\vert _{p/\left(  p-1\right)  }.
\end{align*}

Now suppose that equality holds in (\ref{upbo}). On the one hand, the
conditions for equality in H\"{o}lder's inequality (see, e.g., \cite{HLP88},
p. 24) imply that there exists some $\eta>0$ such that for all $i_{1}%
\in\left[  n_{1}\right]  ,\ldots,i_{r}\in\left[  n_{r}\right]  $%
\[
\left\vert a_{i_{1},\ldots,i_{r}}\right\vert ^{p/\left(  p-1\right)  }%
=\eta|x_{i_{1}}^{\left(  1\right)  }|^{p}\cdots|x_{i_{r}}^{\left(  r\right)
}|^{p},
\]
and so,
\[
\left\vert a_{i_{1},\ldots,i_{r}}\right\vert =\eta^{p-1}|x_{i_{1}}^{\left(
1\right)  }|^{p-1}\cdots|x_{i_{r}}^{\left(  r\right)  }|^{p-1}.
\]
\qquad

On the other hand, equality holds in (\ref{ub1}), and so the complex arguments
of all nonzero terms $a_{i_{1},\ldots,i_{r}}\overline{x_{i_{1}}^{\left(
1\right)  }}\cdots\overline{x_{i_{r}}^{\left(  r\right)  }}$ are the same,
that is to say, there exists $c\in\left[  0,2\pi\right)  $ such that
\[
\arg(a_{i_{1},\ldots,i_{r}}\overline{x_{i_{1}}^{\left(  1\right)  }}%
\cdots\overline{x_{i_{r}}^{\left(  r\right)  }})=c
\]
whenever $a_{i_{1},\ldots,i_{r}}\overline{x_{i_{1}}^{\left(  1\right)  }%
}\cdots\overline{x_{i_{r}}^{\left(  r\right)  }}\neq0.$

For each $k\in\left[  r\right]  ,$ define a vector $\mathbf{y}^{\left(
k\right)  }:=(y_{1}^{\left(  k\right)  },\ldots,y_{n_{k}}^{\left(  k\right)
})$ as
\[
y_{s}^{\left(  k\right)  }:=\left\{
\begin{array}
[c]{ll}%
0, & \text{if }x_{s}^{\left(  k\right)  }=0\text{;}\\
\eta^{\left(  p-1\right)  /r}e^{\left(  c/r\right)  i}x_{s}^{\left(  k\right)
}|x_{s}^{\left(  k\right)  }|^{p-2}, & \text{if }x_{s}^{\left(  k\right)
}\neq0\text{.}%
\end{array}
\right.
\]
Now, if $a_{i_{1},\ldots,i_{r}}\overline{x_{i_{1}}^{\left(  1\right)  }}%
\cdots\overline{x_{i_{r}}^{\left(  r\right)  }}\neq0,$ we see that%
\[
\arg(y_{i_{1}}^{\left(  1\right)  }\cdots y_{i_{r}}^{\left(  r\right)
})=c+\arg(x_{i_{1}}^{\left(  1\right)  }\cdots x_{i_{r}}^{\left(  r\right)
})=\arg(a_{i_{1},\ldots,i_{r}}).
\]
Hence,
\[
a_{i_{1},\ldots,i_{r}}=y_{i_{1}}^{\left(  1\right)  }\cdots y_{i_{r}}^{\left(
r\right)  }%
\]
for all $i_{1}\in\left[  n_{1}\right]  ,\ldots,i_{r}\in\left[  n_{r}\right]
$, and so $A$ is a rank-one matrix.
\end{proof}

In particular, Theorem \ref{tub} implies that
\[
\left\Vert A\right\Vert _{2}\leq\left\vert A\right\vert _{2}%
\]
for any $r$-matrix $A,$ which was shown (without the case of equality) by
Friedland and Lim in \cite{FrLi16}. On the other hand, Proposition \ref{pro2},
\emph{(a)} states that
\[
\left\Vert A\right\Vert _{1}=|A|_{\max}=\left\vert A\right\vert _{\infty},
\]
so (\ref{upbo}) holds for $p=1$ as well, but not the characterization of equality.

Next, we prove some lower bounds on $\left\Vert A\right\Vert _{p}$:

\begin{theorem}
\label{th4}Let $p\geq1$ and $A$ be an $r$-matrix of order $n_{1}\times
\cdots\times n_{r}.$

(a) For every $k\in\left[  r\right]  ,$%
\begin{equation}
\left\Vert A\right\Vert _{p}\geq\frac{1}{n_{1}^{1/p}\cdots n_{r}^{1/p}}%
\sum_{j\in\left[  n_{k}\right]  }\left\vert \Sigma A_{j}^{\left(  k\right)
}\right\vert \text{;} \label{lobo}%
\end{equation}

(b) If $A$ is nonnegative, $p>1,$ and equality holds in (\ref{lobo}) for all
$k\in\left[  r\right]  ,$ then $A$ is regular;

(c) If $p\geq r$ and $A$ is nonnegative, then equality holds in (\ref{lobo})
for all $k\in\left[  r\right]  $ if and only if $A$ is regular.
\end{theorem}

\begin{proof}
\emph{(a) }We outline the proof of (\ref{lobo}) for $k=r$; for other values of
$k$ the proof is similar.

Letting $\mathbf{x}^{\left(  k\right)  }:=n_{k}^{-1/p}\mathbf{j}_{n_{k}}$ for
each $k\in\left[  r-1\right]  ,$ we see that
\[
|\mathbf{x}^{\left(  1\right)  }|_{p}=\cdots=|\mathbf{x}^{\left(  r-1\right)
}|_{p}=1.
\]
Now, for every $j\in\left[  n_{r}\right]  $, set
\[
y_{j}=\Sigma A_{j}^{\left(  r\right)  }=\sum_{i_{1},\ldots,i_{r-1}}%
a_{i_{1},\ldots,i_{r-1,j}},
\]
and define $\mathbf{x}^{\left(  r\right)  }:=(x_{1}^{\left(  r\right)
},\ldots,x_{n_{r}}^{\left(  r\right)  })$ by
\[
x_{j}^{\left(  r\right)  }=\left\{
\begin{array}
[c]{ll}%
n_{r}^{-1/p}y_{j}/\left\vert y_{j}\right\vert , & \text{if }y_{j}\neq
0\text{;}\\
n_{r}^{-1/p}, & \text{if }y_{j}=0.
\end{array}
\right.
\]
Clearly, $|\mathbf{x}^{\left(  r\right)  }|_{p}=1$; thus
\[
\left\Vert A\right\Vert _{p}\geq\sum_{i_{1},\ldots,i_{r}}a_{i_{1},\ldots
,i_{r}}\overline{x_{i_{1}}^{\left(  1\right)  }}\cdots\overline{x_{i_{r}%
}^{\left(  r\right)  }}=\frac{1}{n_{1}^{1/p}\cdots n_{r}^{1/p}}\sum
_{j\in\left[  n_{k}\right]  }\left\vert y_{j}\right\vert =\sum_{j\in\left[
n_{k}\right]  }|\Sigma A_{j}^{\left(  k\right)  }|.
\]

\emph{(b)} Suppose that $A$ is nonnegative and that equality holds in
(\ref{lobo}) for every $k$. Clearly, letting $\mathbf{x}^{\left(  k\right)
}:=n_{k}^{-1/p}\mathbf{j}_{n_{k}}$ for each $k\in\left[  r\right]  ,$ we
obtain an eigenkit to $\left\Vert A\right\Vert _{p},$ because%
\[
L_{A}(\mathbf{x}^{\left(  1\right)  },\ldots,\mathbf{x}^{\left(  r\right)
})=\frac{1}{n_{1}^{1/p}\cdots n_{r}^{1/p}}\Sigma A=\frac{1}{n_{1}^{1/p}\cdots
n_{r}^{1/p}}\sum_{j\in\left[  n_{k}\right]  }\left\vert \Sigma A_{j}^{\left(
k\right)  }\right\vert
\]
for any $k\in\left[  n\right]  $. It turns out that $\mathbf{x}^{\left(
1\right)  },\ldots,\mathbf{x}^{\left(  r\right)  }$ are a solution to the
constrained optimization problem%
\[
\max L_{A}(\mathbf{x}^{\left(  1\right)  },\ldots,\mathbf{x}^{\left(
r\right)  }),
\]
subject to
\[
|\mathbf{x}^{\left(  1\right)  }|_{p}=\cdots=|\mathbf{x}^{\left(  r\right)
}|_{p}=1\text{ and }\mathbf{x}^{\left(  1\right)  }\geq0,\ldots,\mathbf{x}%
^{\left(  r\right)  }\geq0.
\]
Now, Lagrange's method implies that for any $k\in\left[  r\right]  ,$ there
exists a $\mu_{k}$ such that for every $s\in\left[  n_{k}\right]  $,%
\[
\mu_{k}(x_{s}^{\left(  k\right)  })^{p-1}=\sum\{a_{i_{1},\ldots,i_{r}}%
x_{i_{1}}^{\left(  1\right)  }\cdots x_{i_{k-1}}^{\left(  k-1\right)
}x_{i_{k+1}}^{\left(  k+1\right)  }\cdots x_{i_{r}}^{\left(  r\right)
}:\text{ }i_{k}=s,\text{ }i_{j}\in\left[  n_{j}\right]  \text{ for }j\neq k\}
\]
Hence, for any $s\in\left[  n_{k}\right]  ,$ we find that
\begin{align*}
\mu_{k}n_{k}^{-1}  &  =\frac{1}{n_{1}^{1/p}\cdots n_{r}^{1/p}}\sum
\{a_{i_{1},\ldots,i_{r}}:\text{ }i_{k}=s,\text{ }i_{j}\in\left[  n_{j}\right]
\text{ for }j\neq k\}\\
&  =\frac{1}{n_{1}^{1/p}\cdots n_{r}^{1/p}}\Sigma A_{s}^{\left(  k\right)  },
\end{align*}
and therefore,%
\[
\Sigma A_{1}^{\left(  k\right)  }=\cdots=\Sigma A_{n_{k}}^{\left(  k\right)
}.
\]
This proves \emph{(b).}

\emph{(c)} Suppose that $A$ is regular. If $p=r,$ inequality (\ref{HLP})
yields%
\[
\left\Vert A\right\Vert _{r}\leq\frac{1}{n_{1}^{1/r}\cdots n_{r}^{1/r}}\Sigma
A=\frac{1}{n_{1}^{1/r}\cdots n_{r}^{1/r}}\sum_{j\in\left[  n_{k}\right]
}\left\vert \Sigma A_{j}^{\left(  k\right)  }\right\vert .
\]
Hence, equality holds in (\ref{lobo}) for all $k\in\left[  r\right]  $. If
$p>r,$ Proposition \ref{pro2}, clause \emph{(c)} implies that
\[
\left(  n_{1}\cdots n_{r}\right)  ^{1/p}\left\Vert A\right\Vert _{p}%
\leq\left(  n_{1}\cdots n_{r}\right)  ^{1/r}\left\Vert A\right\Vert _{r}%
\leq\sum_{j\in\left[  n_{k}\right]  }\left\vert \Sigma A_{j}^{\left(
k\right)  }\right\vert ,
\]
and so equality holds in (\ref{lobo}) for all $k\in\left[  r\right]  ,$
completing the proof of Theorem \ref{th4}.
\end{proof}

Bound (\ref{lobo}) is quite efficient for some classes of nonnegative
matrices, like $\left(  0,1\right)  $-matrices; in particular, if a $\left(
0,1\right)  $-matrix $A$ has no zero slices, then (\ref{lobo}) never worse
than the similar bound of Friedland and Lim (\cite{FrLi16a}, Lemma 9.1):
\[
\left\Vert A\right\Vert _{2}\geq\left\vert A\right\vert _{\max}\geq\frac
{1}{\left(  n_{1}\cdots n_{r}\right)  ^{1/2}}\left\vert A\right\vert _{2},
\]
and could be better than the more complicated version of Li \cite{ZLi15}.
However, bound (\ref{lobo}) is ill-suited to matrices with small slice sums;
e.g., if the slice sums are zero, then bound (\ref{lobo}) is vacuous. Thus, we
state another tight simple bound, whose proof is omitted:

\begin{proposition}
If $A$ is a matrix and $p>1$, then
\[
\left\Vert A\right\Vert _{p}\geq\max\{\left\vert F\right\vert _{p/\left(
p-1\right)  }:F\text{ is a fiber of }A\}.
\]
If all entries of $A$ are zero except the entries of\ single fiber, then
equality holds.
\end{proposition}

Recall that if $A$ is an $m\times n$ $2$-matrix, then $\left\Vert A\right\Vert
_{2}^{2}$ salsifies the following inequalities%
\[
\left\Vert A\right\Vert _{2}^{2}\geq\frac{1}{m}\sum_{i\in\left[  m\right]
}\left\vert \sum_{j\in\left[  n\right]  }a_{i,j}\right\vert ^{2},\text{
\ \ \ }\left\Vert A\right\Vert _{2}^{2}\geq\frac{1}{n}\sum_{j\in\left[
n\right]  }\left\vert \sum_{i\in\left[  m\right]  }a_{i,j}\right\vert ^{2}.
\]
The purpose of the next theorem is to generalize these bounds to
hypermatrices\footnote{The same result has been recently proved for
hypergraphs and $p=r$ by Liu, Kang, and Shan \cite{LKS16}. Their proof is very
close to the proof of Theorem \ref{th4}, as is also ours.}.

\begin{theorem}
\label{th5}If $p>1$ and $A$ is an $r$-matrix of order $n_{1}\times\cdots\times
n_{r},$ then for every $k\in\left[  r\right]  ,$%
\begin{equation}
\left\Vert A\right\Vert _{p}\geq\left(  \frac{n_{k}^{1/\left(  p-1\right)  }%
}{\left(  n_{1}\cdots n_{r}\right)  ^{1/\left(  p-1\right)  }}\sum
_{j\in\left[  n_{k}\right]  }|A_{j}^{\left(  k\right)  }|^{p/\left(
p-1\right)  }\right)  ^{\left(  p-1\right)  /p}\text{.} \label{lobo1}%
\end{equation}

\end{theorem}

\begin{proof}
We give the proof of (\ref{lobo}) for $k=r$; for other values of $k$ the proof
is essentially the same. Define $|\mathbf{x}^{\left(  1\right)  }%
|,\cdots,|\mathbf{x}^{\left(  r-1\right)  }|,|\mathbf{x}^{r}|$ as follows.

First, let $\mathbf{x}^{\left(  k\right)  }:=n_{k}^{-1/p}\mathbf{j}_{n_{k}}$
for each $k\in\left[  r-1\right]  .$ Clearly,
\[
|\mathbf{x}^{\left(  1\right)  }|_{p}=\cdots=|\mathbf{x}^{\left(  r-1\right)
}|_{p}=1.
\]
Now, let
\[
S:=\left(  \sum_{j\in\left[  n_{r}\right]  }\left\vert A_{j}^{\left(
r\right)  }\right\vert ^{p/\left(  p-1\right)  }\right)  ^{1/p}%
\]
If $S=0,$ then (\ref{lobo1}) is obvious, so we shall assume that $S>0.$ For
every $j\in\left[  n_{r}\right]  $, set
\[
x_{j}:=\left\{
\begin{array}
[c]{ll}%
0, & \text{if }A_{j}^{\left(  r\right)  }=0\text{;}\\
A_{j}^{\left(  r\right)  }\left\vert A_{j}^{\left(  r\right)  }\right\vert
^{-1+1/\left(  p-1\right)  }/S\text{,} & \text{otherwise}.
\end{array}
\right.
\]
Clearly the vector $\mathbf{x}^{\left(  r\right)  }:=(x_{1}^{\left(  r\right)
},\ldots,x_{n_{r}}^{\left(  r\right)  })$ satisfies
\[
\left\vert \mathbf{x}^{\left(  r\right)  }\right\vert _{p}=\left(  \sum
_{j\in\left[  n_{r}\right]  }\left\vert A_{j}^{\left(  r\right)  }\right\vert
^{p/\left(  p-1\right)  }/S^{p}\right)  ^{1/p}=1.
\]
Further,
\begin{align*}
\left\Vert A\right\Vert _{p}  &  \geq\sum_{i_{1},\ldots,i_{r}}a_{i_{1}%
,\ldots,i_{r}}\overline{x_{i_{1}}^{\left(  1\right)  }}\cdots\overline
{x_{i_{r}}^{\left(  r\right)  }}=\frac{n_{r}^{1/p}}{n_{1}^{1/p}\cdots
n_{r}^{1/p}}\sum_{j\in\left[  n_{r}\right]  }\sum_{i_{1},\ldots,i_{r}}%
a_{i_{1},\ldots,i_{r-1}j}\overline{x_{j}}=\frac{n_{r}^{1/p}}{n_{1}^{1/p}\cdots
n_{r}^{1/p}}\sum_{j\in\left[  n_{r}\right]  }A_{j}^{\left(  r\right)
}\overline{x_{j}}\\
&  =\frac{n_{r}^{1/p}}{n_{1}^{1/p}\cdots n_{r}^{1/p}S}\sum_{j\in\left[
n_{r}\right]  }A_{j}^{\left(  r\right)  }\overline{A_{j}^{\left(  r\right)  }%
}\left\vert A_{j}^{\left(  r\right)  }\right\vert ^{-1+1/\left(  p-1\right)
}\\
&  =\frac{n_{r}^{1/p}}{n_{1}^{1/p}\cdots n_{r}^{1/p}S}\sum_{j\in\left[
n_{r}\right]  }\left\vert A_{j}^{\left(  r\right)  }\right\vert ^{p/\left(
p-1\right)  }=\frac{n_{r}^{1/p}}{n_{1}^{1/p}\cdots n_{r}^{1/p}S}S^{p}%
=\frac{n_{r}^{1/p}}{n_{1}^{1/p}\cdots n_{r}^{1/p}}S^{p-1}.
\end{align*}
This completes the proof of (\ref{lobo1}).
\end{proof}

\section{\label{gras}Bounds on the $p$-spectral radius of graphs}

Given a nonempty set $V,$ write $V^{\left(  r\right)  }$ for the family of
all\ $r$-subsets of $V.$ An $r$\emph{-graph} consists of a set of
\emph{vertices} $V=V\left(  G\right)  $ and a set of \emph{edges} $E\left(
G\right)  \subset V^{\left(  r\right)  }.$ It is convenient to identify $G$
with the indicator function of $E\left(  G\right)  $, that is to say,
$G:V^{\left(  r\right)  }\rightarrow\left\{  0,1\right\}  $ and $G\left(
e\right)  =1$ if and only if $e\in E\left(  G\right)  .$ The \emph{order}
$v\left(  G\right)  $ of $G$ is the cardinality of $V$.

More generally, a \emph{weighted }$r$\emph{-graph} $G$ with vertex set $V$ is
a function $G:V^{\left(  r\right)  }\rightarrow\left[  0,\infty\right)  ,$
with edge set defined as $E\left(  G\right)  =\{e:e\in V^{\left(  r\right)  }$
and $G\left(  e\right)  >0\}.$ If $e\in$ $E\left(  G\right)  ,$ then $G\left(
e\right)  $ is called the \emph{weight} of $e,$ which by definition is positive.

Given a weighted $r$-graph $G$ with $V\left(  G\right)  =\left[  n\right]  $,
the \emph{adjacency matrix }$A\left(  G\right)  $ of $G$ is the $r$-matrix of
order $n,$ whose entries are defined by
\begin{equation}
a_{i_{1},\ldots,i_{r}}:=\left\{
\begin{array}
[c]{ll}%
G\left(  i_{1},\ldots,i_{r}\right)  , & \text{if }\left\{  i_{1},\ldots
,i_{r}\right\}  \in E\left(  G\right)  \text{; }\\
0, & \text{otherwise.}%
\end{array}
\right.  \label{AM}%
\end{equation}
Note that $A\left(  G\right)  $ is symmetric and nonnegative. In particular,
if $G$ is unweighted, then $A\left(  G\right)  $ is a $\left(  0,1\right)
$-matrix\footnote{The choice of $0$ and $1$ provides a solid base for weigthed
graphs. Other choices as in \cite{CoDu11} lead to ambiguity as to what the
weight of an edge is.}. We set $\left\vert G\right\vert _{p}:=\left\vert
A\left(  G\right)  \right\vert _{p}$, $\eta^{\left(  p\right)  }\left(
G\right)  :=\eta^{\left(  p\right)  }\left(  A\left(  G\right)  \right)  $,
and $\left\Vert G\right\Vert _{p}:=\left\Vert A\left(  G\right)  \right\Vert
_{p}.$ Since $A\left(  G\right)  $ is symmetric and nonnegative matrix
$\eta^{\left(  r\right)  }\left(  G\right)  =\rho\left(  A\left(  G\right)
\right)  $; we set for short $\rho\left(  G\right)  =\rho\left(  A\left(
G\right)  \right)  $ and call $\rho\left(  G\right)  $ the \emph{spectral
radius }\ of $G$.

A graph $G$ is called $k$\emph{-partite}\ if its\textbf{ }vertices can be
partitioned into $k$ sets so that no edge has two vertices from the same set.

Given a weighted $r$-graph $G,$ and a vertex $v\in V\left(  G\right)  ,$ the
sum%
\[
d\left(  v\right)  :=\{\sum G\left(  e\right)  :e\in E\left(  G\right)  \text{
\ and \ }v\in e\}
\]
is called the \emph{degree }of $v.$ A graph $G$ is called \emph{regular} if
the degrees of its vertices are equal. An $r$-partite $r$-graph is called
\emph{semiregular}, if all vertices belonging to the same partition set have
the same degree.

A weighted $r$-partite graph is called \emph{rank-one }if each vertex $u$ can
be assigned a real number $x_{u}$ such that $G\left(  i_{1},\ldots
,i_{r}\right)  =x_{i_{1}}\cdots x_{i_{r}}$ for every edge $\left\{
i_{1},\ldots,i_{r}\right\}  \in E\left(  G\right)  .$

To the end of this section we list several new theorems about hypergraphs,
which follow from the above results about hypermatrices. As mentioned before,
other similar results can be found in \cite{Nik14} and its references.\medskip

Theorem \ref{th1} implies an extension of a result of\ Bhattacharya,
Friedland, and Peled (\cite{BFP08}, p. 4) to $r$-partite $r$-graphs:

\begin{theorem}
Let $p\geq1$, let $G$ be a weighted $r$-partite $r$-graph of order $n,$ and
let $\left[  n\right]  =N_{1}\cup\cdots\cup N_{r}$ be its partition. If
$\mathbf{x}$ is an eigenvector to $\eta^{\left(  p\right)  }\left(  A\right)
$, then for every $i\in\left[  r\right]  ,$ the vector $\mathbf{x}|_{N_{i}}$
satisfies
\[
\left\vert \mathbf{x}|_{N_{i}}\right\vert _{p}=r^{-1/p}.
\]

\end{theorem}

Theorems \ref{th2} and \ref{tub} imply the following upper bound that
partially generalizes a result of Nosal \cite{Nos70} to $r$-graphs.

\begin{theorem}
Let $p\geq1.$ If $G$ is a weighted $r$-partite $r$-graph, then%
\[
\frac{r^{r/p}}{r!}\eta^{\left(  p\right)  }\left(  G\right)  \leq\left\Vert
G\right\Vert _{p}\leq|G|_{p/\left(  p-1\right)  }.
\]
Equality holds if and only if $G$ is a rank-one graph.
\end{theorem}

The proof of Theorem \ref{th4} can be adapted to yield the following lower
bound on $\eta^{\left(  p\right)  }\left(  G\right)  $:

\begin{theorem}
Let $p\geq1$ and let $G$ be a weighted $r$-partite $r$-graph. If $n_{1}%
,\ldots,n_{r}$ are the sizes of its partition sets, then,%
\begin{equation}
\eta^{\left(  p\right)  }\left(  G\right)  \geq\frac{r!/r^{r/p}}{n_{1}%
^{1/p}\cdots n_{r}^{1/p}}\left\vert G\right\vert _{1}. \label{lobg}%
\end{equation}
If $p>1$ and equality holds in (\ref{lobg}), then $G$ is semiregular. If
$p\geq r$, then equality holds in (\ref{lobg}) if and only if $G$ is semiregular.
\end{theorem}

Theorems \ref{mth1} and \ref{th3} imply the following extension of a result of
Favaron, Mah\'{e}o, and Sacl\'{e} (\cite{FMS93}, Corollary 2.3) to $r$-graphs:

\begin{theorem}
Let $G$ be a weighted $r$-graph of order $n$. If $d_{1},\ldots,d_{n}$ are the
degrees of $G,$ then
\[
\rho\left(  G\right)  ^{r}\leq(r-1)!^{r}\max_{k\in\left[  n\right]  }\sum
G\left(  k,i_{2},\ldots,i_{r}\right)  d_{i_{2}}\cdots d_{i_{r}}.
\]

\end{theorem}

Theorems \ref{mth1} and \ref{mth} imply the following extension of a result of
Berman and Zhang (\cite{BeZh01}, Lemma 2.1) to $r$-graphs:

\begin{theorem}
Let $G$ be a weighted $r$-graph of order $n$. If $d_{1},\ldots,d_{n}$ are the
degrees of $G,$ then
\[
\rho\left(  G\right)  ^{r}\leq(r-1)!^{r}\max_{\left\{  i_{1},\ldots
,i_{r}\right\}  \in E\left(  G\right)  }d_{i_{1}}\cdots d_{i_{r}}.
\]

\end{theorem}

Theorem \ref{th5} implies an analog of Hofmeister's bound on the spectral
radius of graphs and extends the main result of \cite{LKS16}, which is the
case $p=r.$

\begin{theorem}
Let $G$ be a weighted $r$-graph of order $n$ and $d_{1},\ldots,d_{n}$ be the
degrees of $G$. If $p\geq r,$ then%
\[
\eta^{\left(  p\right)  }\left(  G\right)  \geq\frac{r!}{r^{r/p}}\left(
\frac{1}{n^{\left(  r-1\right)  /\left(  p-1\right)  }}\sum_{i\in\left[
n\right]  }d_{i}^{p/\left(  p-1\right)  }\right)  ^{\left(  p-1\right)  /p}.
\]

\end{theorem}

\bigskip

\textbf{Concluding remark}

It is well known that analytic methods can be applied to combinatorial
problems. This paper may be regarded as a demonstration of the inverse application.

\end{document}